\documentclass[11pt,reqno]{amsart}
\usepackage{amssymb,amsthm,amsfonts,amstext}
\usepackage{amsmath}
\usepackage{graphicx}
\usepackage[american]{babel}
\usepackage[ansinew]{inputenc}
\makeatletter \@addtoreset{equation}{section} \makeatother

\renewcommand\thetable{\thesection.\@arabic\c@table}
\newtheorem{theorem}{Theorem}[section]
\newtheorem{lemma}[theorem]{Lemma}
\newtheorem{proposition}[theorem]{Proposition}

\newtheorem{definition}{Definition}
\newtheorem{remark}{Remark}[section]

\usepackage{a4wide}

\title[H. L. for a Particle System with degenerate rates]{Hydrodynamic
  Limit for a Particle System \\ with degenerate rates}

\date{\today}

\author{Gonçalves, P.}

\address{IMPA, Estrada Dona Castorina 110, CEP 22460 Rio de Janeiro, Brasil}
\email{patg@impa.br}

\author{Landim, C.}

\address{IMPA, Estrada Dona Castorina 110,
CEP 22460-320 Rio de Janeiro, Brasil\\
CNRS UMR 6085, Universit\'e de Rouen, UMR 6085, Avenue de
l'Universit\'e, BP.12, Technop\^ole du Madrillet, F76801
Saint-\'Etienne-du-Rouvray, France}
\email{landim@impa.br}

\author{Toninelli, C.}

\address{ Laboratoire de Probabilit{\'e}s et Mod{\`e}les
  Al{\`e}atoires CNRS-UMR 7599, Univ.Paris VI-VII, 4 Pl.Jussieu,
Paris, FRANCE} \email{ctoninel@ccr.jussieu.fr}

\begin{document}

\begin{abstract}

  We study the hydrodynamic limit for some conservative particle
  systems with degenerate rates, namely with nearest neighbor
  exchange rates which vanish for certain configurations.  These
  models belong to the class of {\sl kinetically constrained lattice
    gases} (KCLG) which have been introduced and intensively studied
  in physics literature as simple models for the liquid/glass
  transition.  Due to the degeneracy of rates for KCLG there exists
  {\sl blocked configurations} which do not evolve under the dynamics
  and in general the hyperplanes of configurations with a fixed number
  of particles can be decomposed into different irreducible sets.  As
  a consequence, both the Entropy and Relative Entropy method cannot
  be straightforwardly applied to prove the hydrodynamic limit.  In
  particular, some care should be put when proving the One and Two
  block Lemmas which guarantee local convergence to equilibrium.
We show that,
for initial profiles smooth enough and bounded away from zero and one,
the macroscopic density
profile for our KCLG evolves under the diffusive time scaling
according to the porous medium equation.
Then we prove the same result for more general profiles
for a slightly perturbed dynamics
obtained by adding jumps of the Symmetric Simple
Exclusion.
The role of the latter is to remove the degeneracy of rates and at the same time they
are properly slowed down in order not to change the macroscopic behavior.
The equilibrium fluctuations and the magnitude of the spectral gap for
this perturbed model are also obtained.
\end{abstract}
\renewcommand{\subjclassname}{\textup{2000} Mathematics Subject Classification}
\begin{thanks} {The first author wants to express her gratitude to
    F.C.T. (Portugal) for supporting her Phd with the grant /SFRH/ BD/ 11406/ 2002. The third author thanks IMPA for hospitality and
CNPq for supporting this work via the grant Pronex E-26/151.943/2004}
\end{thanks}
\maketitle
\section{Introduction}

The purpose of this article is to define a conservative interacting particle system whose macroscopic density profile evolves according to the
\textbf{porous medium equation}, namely the partial differential equation given by
\begin{equation}
\begin{cases}
\partial_{t}\rho(t,u)=\Delta\rho^{m}(t,u) \\ \label{eq:porous m}
\rho(0,\cdot) = \rho_0(\cdot)
\end{cases}
\end{equation}
where $\Delta=\sum_{1\leq{j}\leq{d}}\partial_{u_{j}}^{2}$ and $m\in{\mathbb{N}\setminus\{1\}}$. This can be rewritten in the divergence form as
$\partial_{t}\rho(t,u)=\nabla(D(\rho(t,u))\nabla(\rho(t,u)))$ with diffusion coefficient $D(\rho(t,u))=m\rho^{m-1}(t,u)$. Note that $D(\rho)$ goes to
zero as $\rho\rightarrow{0}$, thus the equation
looses its parabolic character.

One of the most important properties  of the porous medium equation is that its solutions can be compactly supported at each fixed time or, in
physical terms, that is has a finite speed of propagation. This is in strong contrast with the solutions of the classical heat equation. A
non negative solution of the heat equation is always positive on its domain. A second observation is that the solutions of the equation
(\ref{eq:porous m}) can be continuous on the domain of definition without being smooth at the boundary. The existence of this kind of solutions
is a direct consequence of the degeneracy of $D(\rho)$ as $\rho\to 0$. For a reference on the mathematical properties of the porous medium equation we refer
to \cite{V.} and references therein. We also recall that such equation is relevant in different contexts in physical literature beyond the original motivation of describing the density of an ideal gas flowing isothermally through an homogeneous porous medium (corresponding to the choice $m=2$).

A microscopic derivation of the porous medium equation has been already obtained in \cite{F.I.S.} by considering a model in which the occupation
number is a continuous variable. Here we study instead models with discrete occupation variables: our microscopic dynamics are given by
stochastic lattice gases with hard core exclusion, namely systems of interacting particles on the d-dimensional discrete torus
${\mathbb{T}_{N}^d}$ with the constraint that on each site there can be at most one particle.  A configuration is therefore defined by giving
for each site $x\in\mathbb{T}_{N}^d$ the occupation variable, $\eta(x)\in\{0,1\}$, which stands for empty or occupied sites, respectively.
Evolution is then given by a continuous time Markov process during which the jump of a particle from a site $x$ to a nearest neighbor site $y$
occurs at rate $c(x,y,\eta)$. The choice $c(x,y,\eta)=1$ corresponds to the Symmetric Simple Exclusion process (SSEP) and, as is very well
known, leads to the heat equation under diffusive re-scaling of time, namely $D(\rho)=1$. In order to slow down the low density dynamics and
obtain a diffusion coefficient which degenerates for $\rho\to 0$, we impose a local constraint (in addition to hard core exclusion) that must be
satisfied in order for a particle jump to be allowed.  This constraint is imposed at the level of rates: for any nearest neighbor couple $(x,y)$
we fix the exchange rate $c(x,y,\eta)$ to be zero if $\eta$ does not satisfy a local
 constraint. The latter corresponds to
requiring a minimal number of occupied sites in a proper
neighborhood of $(x,y)$. Since the typical number
of particles in a given region is monotone with $\rho$, $D(\rho)$
will decrease as $\rho$ is decreased.
At the same time the rates are chosen in order to
satisfy the detailed balance condition with respect to Bernoulli
product measure at any density (as for SSEP), namely the constraints do
not add further interactions beyond hard core exclusion.
The models we introduce belong to the class of {\sl kinetically constrained lattice gases} (KCLG), which have been introduced and analyzed in
physical literature since the late 1980's to model liquid/glass and more general jamming transitions (see for a review \cite{R.S., Bi.T.} and
references therein). In this context, the constraints are devised to mimic the fact that the motion of a particle in a dense medium (e.g. a molecule
 in a low temperature liquid) can be inhibited by the geometrical constraints induced by the surrounding particles \footnote{Note that the role of
 particles and vacancy is usually
exchanged in physical literature with respect to our convention: vacancies rather then particles are needed to facilitate motion. With this
notation the diffusion coefficient degenerates at high rather than low density.}. For most KCLG a degenerate diffusion coefficient is expected
when $\rho\to 0$ but, for very restrictive choices of the
constraints, the degeneracy could even occur at non trivial critical density.

Here we provide the first derivation of the hydrodynamic limit for the simplest  KCLG, the so called {\sl non cooperative} KCLG. This means that
the rates are such that it is possible to construct a proper finite group of particles, the {\sl mobile cluster},  which has the following
properties. There exists (at least one) allowed   (i.e. with strictly positive rates) sequence of nearest neighbor jumps which allows to shift
the mobile cluster to any other position. Furthermore, this allowed path should be deterministic, i.e. independent on the value of the
occupation variables on the remaining sites. Finally, the jump of any other particle to a neighboring site should be allowed when the mobile
cluster is brought in a proper position in its vicinity. Therefore a configuration containing the mobile cluster can be connected to any other
one with the same property by an allowed path. This, very loosely speaking, means that non-cooperative KCLG should behave like a re-scaled SSEP
with the mobile clusters playing the role of single particles for SSEP. Thus their diffusion coefficient should degenerate only for $\rho\to 0$,
when the density of mobile clusters goes to zero as power law with $\rho$. Indeed, in our proof of the latter result, the non cooperative
property will play a key role since we will use it to provide paths which allow to perform particle exchanges. The use of similar path arguments
for KCLG had already been exploited in \cite{B.T.}, were the scaling with the lattice size of the spectral gap and log Sobolev constant for non
cooperative models in contact with particle reservoirs at the boundary were derived. A similar case in which the diffusion coefficient does not
vanish has already been studied in \cite{HS.}. Finally, we stress that all along our proofs of hydrodynamics (both for the Entropy and Relative
Entropy method) we use an additional property of the rates: the fact that they are of gradient type. A natural development of our work would be
to generalize the present results to non cooperative KCLG with non gradient rates.


An outline of the paper follows. In section \ref{results} we introduce some notation, define our models and state the main results. The models
we consider are either truly KCLG or perturbed models in which we add proper jumps of SSEP to remove the degeneracy of the exchange rates. In
section \ref{relativeentropy} the proof of the hydrodynamic limit via the Relative Entropy method for the KCLG is presented. The proof of the
hydrodynamic limit for the perturbed models via the Entropy method is described in section \ref{entropy}. The proof of a Replacement Lemma
needed for the Entropy method is postponed to the section \ref{replacement}. Finally, in section \ref{gap}, we study the spectral gap for both
the unperturbed and perturbed dynamics.

\section{Statement of results}
\label{results}

The models we consider are continuous time Markov processes $\eta_{t}$ with space state $\chi^N_{d}=\{0,1\}^{\mathbb{T}_{N}^d}$, where
${\mathbb{T}_{N}^d}=\{0,1,..,N-1\}^d$ is the discrete d-dimensional torus. Let $\eta$ denote a configuration
 in $\chi_{d}^N$, $x$ a site in $\mathbb{T}_{N}^d$ and $\eta(x)=1$ if there is a particle at site $x$, otherwise $\eta(x)=0$.
The elementary moves which occur during evolution correspond to jump of particles among nearest neighbors, $x$ and $y$, occurring at a rate
$c(x,y,\eta)$ which depends both on the couple $(x,y)$ and on the value of the configuration $\eta$ in a finite neighborhood of $x$ and $y$.
Furthermore these rates are symmetric with respect of an $x$-$y$ exchange $c(x,y,\eta)=c(y,x,\eta)$ and are translation invariant. More
precisely the dynamics is defined by means of an infinitesimal generator acting on local functions $f:\chi_{d}^N\rightarrow{\mathbb{R}}$ as
\begin{equation}\label{eq:lpgenerator}
(\mathcal{L}_{P}f)(\eta)=\sum_{\substack{x,y\in{\mathbb{T}_{N}^d}\\|x-y|=1}}c(x,y,\eta)\eta(x)(1-\eta(y))(f(\eta^{x,y})-f(\eta)),
\end{equation}
where  $|x-y|=\sum_{1\leq{i\leq{d}}}|x_{i}-y_{i}|$ is the sum norm in $\mathbb{R}^d$ and
\begin{eqnarray}
 \eta^{x,y}(z)=\left\{
\begin{array}{rl}
\eta(z), & \mbox{if $z\neq{x,y}$}\\ \label{etaxy}
\eta(y), & \mbox{if $z=x$}\\
\eta(x), & \mbox{if $z=y$}
\end{array}.
\right.
\end{eqnarray}

In the sequel we consider the rates
\begin{equation}
c(x,x+e_{j},\eta)=\eta(x-e_{j})+\eta(x+2e_{j}) \label{rates}
\end{equation}
where $\{e_{j},j=1,..,d\}$ denotes the canonical basis of $\mathbb{R}^d$ and we will prove all the theorems for this choice. This, as we will prove,
 leads in the hydrodynamic limit to the porous medium equation
(\ref{eq:porous m}) for $m=2$.  Also, we can provide for any other $m$
 a proper choice of the
rates such that all proofs can be readily extended leading in the diffusive re-scaling to the porous medium equation with the correspondent $m$.
For instance in the case $m=3$, the jump rates to be considered are
\begin{equation*}
c(x,x+e_{j},\eta)=\eta(x-e_{j})\eta(x+2e_{j})+\eta(x-2e_{j})\eta(x-e_{j})+\eta(x+2e_{j})\eta(x+3e_{j}).
\end{equation*}
Note that both the choices of the jump rates taken above have the property of defining a \textbf{gradient system}, namely one for which the
instantaneous current between the sites $0$ and $e_{j}$:
\begin{equation*}
W_{0,e_{j}}(\eta)=c(0,e_{j},\eta)\{\eta(0)(1-\eta(e_{j}))-\eta(e_{j})(1-\eta(0))\}
\end{equation*}
can be rewritten as a function minus its translation. This property will be a key ingredient when deriving the hydrodynamic limit. Also, both
the models are non cooperative in any dimension according to the definition given in introduction. Consider for example the rates (\ref{rates})
in one dimension. A possible choice for the mobile cluster is given by two particles at distance at most two. Let us describe  the deterministic
sequence of allowed moves (i.e. with strictly positive exchange rate) which we should perform  to shift of one step to the right the mobile
cluster when $\eta(x)=\eta(x+e_1)=1$, i.e. to transform $\eta$ into $\eta'$ with $\eta'(x+e_1)=\eta'(x+2e_1)=1$, $\eta'(x)=\eta(x+2e_1)$ and
$\eta'(z)=\eta(z)$ for $z\not\in(x,x+e_1,x+2e_1)$. First we make the move $\eta\to\eta^{x+e_1,x+2e_1}$, which is allowed since
$c(x+e_1,x+2e_1,\eta)\geq \eta(x)=1$. Then we perform the move $\eta^{x+e_1,x+2e_1}\to(\eta^{x+e_1,x+2e_1})^{x,x+e_1}$ which is also allowed,
since $c(x,x+e_1,\eta^{x+e_1,x+2e_1})\geq \eta^{x+e_1,x+2e_1}(x+2e_1)=\eta(x+e_1)=1$. The case in which we have instead the particles at
distance two, $\eta(x)=\eta(x+2e_1)=1$, can be treated analogously. The second property which characterizes non cooperative models can also be
readily checked: if we are given any two neighboring sites, $(y,y+e_1)$, the exchange of their occupation variables can be performed if we bring
the mobile group of two particles in $(y-2e_1,y-e_1)$ since $c(y,y+e_1,\eta)\geq \eta(y-e_1)$. It is then possible to verify  that any two
configurations $\eta$ and $\eta'$ with the same number of particles, $\sum\eta(x)=\sum\eta'(x)$ and both containing at least two particles at
distance at most two can be connected one to another  via a sequence of allowed jumps.

Let $\nu_{\alpha}$ be the Bernoulli product measure in $\chi^N_{d}$, with
 $\alpha \in{(0,1)}$.  Since $c(x,y,\eta)=c(y,x,\eta)$ $\forall{x,y\in{\mathbb{T}^d_{N}}}$, the measures $\nu_{\alpha}$ are
reversible for this process $\forall \alpha$, as for SSEP. By the degeneracy of the rates, other invariant measures arise naturally. For example in the one
dimensional setting, any configuration $\eta$  such that the distance
between the position of two consecutive occupied sites is bigger than two
has all the exchange rates which vanish. Therefore it is a  {\sl blocked configuration} and a Dirac measure supported on it is
an invariant measure for this process.

Let $\Sigma_{N,k}$ denote the hyperplane of configurations with $k$ particles, namely
\begin{equation} \label{hyperplane}
\Sigma_{N,k}=\{\eta\in{\chi_d^N}:\sum_{x\in{\mathbb{T}_{N}^d}}\eta(x)=k\},
\end{equation}
which is invariant under the dynamics. We say that $\mathcal{O}$ is an irreducible component of $\Sigma_{N,k}$ if for every $\eta$,
$\xi\in{\mathcal{O}}$ it is possible to go from $\eta$ to $\xi$ by a sequence of allowed jumps.
For SSEP, the hyperplanes $\Sigma_{N,K}$ are irreducible components for any choice of $k$ and $N$. In the presence of constraints, a more complicated decomposition in general arises. For example in $d=1$ with the rates (\ref{rates}), the above observation  on blocked configurations and on the non cooperative character  of the model, leads to the following irreducible decomposition for the hyperplanes. If $k>N/3$, $\Sigma_{N,K}$ is irreducible. Instead, if $k\leq N/3$, $\Sigma_{N,k}$ is reducible and decomposable into the irreducible component which contains all configurations with at least one couple of particles at distance at  most two plus many irreducible sets each containing only a blocked configuration

The irreducible decomposition in dimension $d>1$ is more complicated. In this case the model
is still non cooperative and a possible mobile cluster is given by a d-dimensional hypercube of particles of
linear size $2$. For any choice of the spatial dimension $d$, it is possible to identify a constant
 $C(d)<\infty$ such that the hyperplane $\Sigma_{N,k}$ is irreducible for $k> C(d) (N/3)^d$, while
  it is reducible in several components for $k\leq C(d)(N/3)^d$. In this case we have: (i) the irreducible
   component which contain
configurations with at least one d-dimensional hypercube of particles of linear size two plus all configurations that can be connected to these;
(ii) irreducible components which contain single (blocked) configuration; (iii) other irreducible components which contain neither blocked
configurations nor any configuration belonging to (i). An example of irreducible set of the third kind for the rates (\ref{rates}) in $d=2$ is
for example the one that contains all configurations which have two particles at distance smaller or equal to two on a given line, for
$x=(x_{1},x_{2})$ such that $\eta(x+e_i)=\eta(x)=1$ or  $\eta(x+2 e_i)=\eta(x)=1$ and are completely empty $\forall$ $y=(y_{1},y_{2})$ which do
not belong to the same line, namely $y_2\neq x_2$.



In order to investigate the hydrodynamic limit, we need to settle some notation. Define the \textbf{empirical measure} by:
\begin{equation} \label{eq:empirical}
\pi_{t}^{N}(du)=\pi^{N}(\eta_{t},du)=\frac{1}{N^d}\sum_{x\in{\mathbb{T}_{N}^d}}\eta_{tN^2}(x)\delta_{\frac{x}{N}}(du),
\end{equation}
where $\delta_{u}$ denotes the Dirac measure at $u$.

Let $\mathbb{T}^d$ denote the d-dimensional torus. Fix now, a initial profile $\rho_{0}:\mathbb{T}^d\rightarrow{[0,1]}$ and denote by
$(\mu^{N})_{N}$ a sequence of probability measures on $\chi_{d}^N$.

\begin{definition}
A sequence \textbf{$(\mu^{N})_{N}$ is associated to $\rho_{0}$}, if for every continuous function $H:{\mathbb{T}^d}\rightarrow{\mathbb{R}}$ and
for every $\delta>0$
\begin{equation} \label{measureassociatedtoprofile}
\lim_{N\rightarrow{+\infty}}\mu^{N}\Big[\eta:\Big|\frac{1}{N}\sum_{x\in{\mathbb{T}_{N}^d}}H\Big(\frac{x}{N}\Big)\eta(x)-\int_{\mathbb{T}^d}H(u)\rho_{0}(u)du\Big|>\delta\Big]=0.
\end{equation}
\end{definition}

Our goal consists in showing that, if at time $t=0$ the empirical measures are associated to some initial profile $\rho_{0}$, at the
macroscopic time $t$ they are associated to a profile $\rho_{t}$ which is  the solution of the hydrodynamic equation (\ref{eq:porous
m}).

Fix $\epsilon>0$ and let $\rho_{0}:\mathbb{T}^d\rightarrow{[0,1]}$ be a profile of class $C^{2+\epsilon}(\mathbb{T}^d)$. By Theorem A2.4.1 of
\cite{K.L.}, equation (\ref{eq:porous m}) admits a solution that we denote by $\rho(t,\cdot)$ which is of class
$C^{1+\epsilon,2+\epsilon}(\mathbb{R}_{+}\times{\mathbb{T}^d})$.

Here we also have to impose a bound condition on the initial profile, as the existence of a strictly positive constant $\delta_{0}$ such that
\begin{equation} \label{boundcondinitialprofile}
\forall{u\in{\mathbb{T}^d}}, \\ \quad\  \delta_{0}\leq{\rho_{0}(u)}\leq{1-\delta_{0}}.
\end{equation}

Let $\nu^N_{\rho_{0}(\cdot)}$ be the product measure in $\chi_{d}^N$ such that:
\begin{equation*}
\nu^N_{\rho_{0}(\cdot)}\{\eta,\eta(x)=1\}=\rho_{0}(x/N).
\end{equation*}

For two measures $\mu$ and $\nu$ in $\chi_{d}^N$ denote by $H(\mu/\nu)$ the \textbf{relative entropy} of $\mu$ with respect to $\nu$, defined
by:
\begin{equation} \label{df:entropy}
H(\mu/\nu)=\sup_{f}\Big\{\int fd\mu-\log\int e^{f} d\nu\Big\},
\end{equation}
where the supreme is carried over all continuous functions.
\begin{theorem} \label{th:hlrem}
Let $\rho_{0}:\mathbb{T}^d\rightarrow{[0,1]}$ be a initial profile of class $C^{2+\epsilon}(\mathbb{T}^d)$ that satisfies the bound condition
(\ref{boundcondinitialprofile}). Let $(\mu^{N})_{N}$ be a sequence of probability measures on $\chi_{d}^N$ such that:
\begin{equation} \label{entropymunurho0}
H(\mu^{N}/\nu^N_{\rho_{0}(.)})=o(N^d).
\end{equation}
 Then, for each $t\geq{0}$
\begin{equation*}
\pi^{N}_{t}(du)\xrightarrow[N\rightarrow{+\infty}]\,\rho(t,u)du
\end{equation*}
in probability, where $\rho(t,u)$ is a smooth solution of equation (\ref{eq:porous m}).
\end{theorem}

The porous medium equation (\ref{eq:porous m}) presents an interesting behavior for profiles which vanish in some region of their domain. The
previous method does not give us information in this situation since in order to apply it the initial profile has to satisfy condition
(\ref{boundcondinitialprofile}).
When these conditions are not verified, another possibility to derive the hydrodynamic limit is to use the Entropy
method which is due to Guo, Papanicolau and Varadhan \cite{G.P.V.}. However the latter strongly relies on the ergodicity
of the Markov process restricted to an hyperplane which, as we discussed above, does not hold for the process generated by ${\mathcal{L}}_P$.
One way to overcome this problem is to perturb slightly the dynamics in such a way that the frozen
states are destroyed, ergodicity is restored  and the macroscopic hydrodynamic behavior still evolves according to (\ref{eq:porous m}).

More precisely we consider a Markov process with generator given by
\begin{equation} \label{eq:sumgenerator}
\mathcal{L}^N_{\theta}=\mathcal{L}_{P}+N^{\theta-2}\mathcal{L}_{S}
\end{equation}
where $0<\theta<2$, $\mathcal{L}_{P}$ was introduced in (\ref{eq:lpgenerator}) and $\mathcal{L}_{S}$ is the generator of the SSEP, which acts on
local functions $f:\chi_{d}^N\rightarrow{\mathbb{R}}$ as
\begin{equation*}
(\mathcal{L}_{S}f)(\eta)=\sum_{\substack{x,y\in{\mathbb{T}_{N}^d}\\|x-y|=1}}\frac{1}{2d}\eta(x)(1-\eta(y))(f(\eta^{x,y})-f(\eta)),
\end{equation*}
where $\eta^{x,y}$ was defined in (\ref{etaxy}). Since the Bernoulli product measures $\nu_{\alpha}$ are invariant for both the
processes generated by $\mathcal{L}_{P}$ and by $\mathcal{L}_{S}$, they are also invariant for the perturbed process generated by $\mathcal{L}_{\theta}^{N}$.
Thanks to the restriction $0<\theta<2$, the generator of SSEP is slowed down and the Markov process generated by $\mathcal {L}_P$ and $\mathcal{L}_\theta^N$ will have the same hydrodynamic limit. Furthermore, since on any configuration
each nearest neighbour
exchange rate is strictly positive, the Markov process generated by $\mathcal {L}_{\theta}^N$ is ergodic on each hyperplane with fixed particle number and we can apply the Entropy method to derive its hydrodynamic limit.

For a probability measure $\mu$ on $\chi_d^N$, denote by $\mathbb{P}_{\mu}^{\theta,N}=\mathbb{P}_{\mu}$ the probability measure on the space
$D([0,T],\chi_d^N)$, induced by the Markov process with generator $\mathcal{L}_{\theta}^N$ speeded up by $N^2$ and with initial measure $\mu$;
and by $\mathbb{E}_{\mu}$ the expectation with respect to $\mathbb{P}_{\mu}$.

We start by introducing the definition of weak solutions of equation (\ref{eq:porous m}).
\begin{definition}
Fix a bounded profile $\rho_{0}:\mathbb{T}^{d}\rightarrow{\mathbb{R}}$. A bounded function
$\rho:[0,T]\times\mathbb{T}^{d}\rightarrow{\mathbb{R}}$ is a \textbf{weak solution} of equation (\ref{eq:porous m}) if for every function
$H:[0,T]\times\mathbb{T}^{d}\rightarrow{\mathbb{R}}$ of class $C^{1,2}([0,T]\times\mathbb{T}^{d})$

\begin{equation*}
\int_{0}^{T}dt\int_{\mathbb{T}^{d}}du\Big\{\rho(t,u)\partial_{t}H(t,u)+ (\rho(t,u))^2\sum_{1\leq{i}\leq{d}}\partial_{u_{i}}^{2}H(t,u)\Big\}
\end{equation*}
\begin{equation}
+\int_{\mathbb{T}^d}\rho_{0}(u)H(0,u)du=\int_{\mathbb{T}^d}\rho(T,u)H(T,u)du. \label{definitionweaksolution}
\end{equation}
\end{definition}

The Entropy method requires the uniqueness of a weak solution of the hydrodynamic equation. This is a consequence of Theorem A2.4.4 of
\cite{K.L.} together with the fact that there is no more than a particle per site.

\begin{theorem} \label{th:hlrm}
Let $\rho_{0}:\mathbb{T}^d\rightarrow{[0,1]}$ and $(\mu^{N})_{N}$ be a sequence of probability measures on $\chi_{d}^N$ associated to the
profile $\rho_{0}$. Then, for every $t\geq{0}$, for every continuous function $H:\mathbb{T}^d\rightarrow{\mathbb{R}}$ and for every $\delta>0$,
\begin{equation*}
\lim_{N\rightarrow{+\infty}}\mathbb{P}_{\mu^{N}}\Big[\Big|\frac{1}{N^d}\sum_{x\in{\mathbb{T}_{N}^d}}{H\Big(\frac{x}{N}\Big)\eta_{t}(x)}
-\int_{\mathbb{T}^d}{H(u)\rho(t,u)du}\Big|>{\delta}\Big]=0,
\end{equation*}
where $\rho(t,u)$ is the unique weak solution of equation (\ref{eq:porous m}).
\end{theorem}

Once we have established the Law of Large Numbers for the empirical measure for the process with generator $\mathcal{L}_{\theta}^N$, the next
step is to obtain the \textbf{Central Limit Theorem} starting from the invariant measure $\nu_{\rho}$.

For each $z>0$ (resp. $z<0$) define $h_{z}:\mathbb{T}^d\rightarrow\mathbb{R}$ by $h_{z}(u)=\sqrt{2}\cos(2\pi z\cdot{u})$ (resp.
$h_{z}(u)=\sqrt{2}\sin(2\pi z\cdot u)$ and let $h_{0}=1$. Here $\cdot$ denotes the inner product of $\mathbb{R}^d$. It is well known that the
set $\{h_{z},z\in \mathbb{Z}^d\}$ is an orthonormal basis of $L^{2}(\mathbb{T}^d)$. In this space consider the operator $\Omega=1-\Delta$. A
simple computation shows that $\Omega h_{z}=\gamma_{z}h_{z}$ where $\gamma_{z}=1+4\pi^{2}||z||^{2}$.

For a positive integer $k$, denote by $\mathcal{H}_{k}$ the space obtained as the completion of $C^{\infty}(\mathbb{T}^d)$ endowed with the
inner product defined by $<f,g>_{k}=<f,\Omega^{k}g>$. Let $\mathcal{H}_{-k}$ denote the dual of $\mathcal{H}_{k}$ with respect to the inner
product of $L^{2}(\mathbb{T}^{d})$.

Recall the definition of the empirical measure in (\ref{eq:empirical}) where $\eta_{t}$ denotes the Markov process with generator
$\mathcal{L}_{N}^{\theta}$, with $0<\theta<2$. As we want to investigate the fluctuations of this measure, fix $\rho>0$ and denote by
$\mathcal{Y}_{.}^{N}$ the \textbf{density fluctuation field} acting on smooth functions $H$ as
\begin{equation} \label{densityfield}
\mathcal{Y}_{t}^{N}(H)=\frac{1}{N^{d/2}}\sum_{x\in{\mathbb{T}^d_{N}}}H\Big(\frac{x}{N}\Big)(\eta_{tN^2}(x)-\rho).
\end{equation}

Fix a positive integer $k$ and denote by $D([0,T],\mathcal{H}_{-k})$ (resp. $C([0,T],\mathcal{H}_{-k})$) the space of $\mathcal{H}_{-k}$
functions, that are right continuous with left limits (resp. continuous), endowed with the uniform weak topology. Denote by $\mathcal{Q}_{N}$
the probability measure on $D([0,T],\mathcal{H}_{-k})$ induced by $\mathcal{Y}^{N}_{.}$ and $\nu_{\rho}$.
\begin{theorem}
Fix an integer $k\geq{3}$. Denote by $\mathcal{Q}$ be the probability measure on $C([0,T],\mathcal{H}_{-k})$ corresponding to the stationary
generalized Ornstein-Uhlenbeck process with mean $0$ and covariance given by
\begin{equation*}
E_{\mathcal{Q}}[Y_{t}(H)Y_{s}(G)]=\frac{\textbf{Var}(\nu_{\rho},\eta(0))}{\sqrt{8\pi \rho(t-s)}}\int_{\mathbb{R}^d}du \int_{\mathbb{R}^d}dv
\bar{H}(u)\bar{G}(v)\exp\Big\{-\frac{(u-v)^{2}}{8(t-s)\rho}\Big\}
\end{equation*}
for every $0\leq{s}\leq{t}$ and $H$, $G$ in $\mathcal{H}_{k}$. Here $\bar{H}$ and $\bar{G}$ are periodic functions equal to $H$, $G$ in
$\mathbb{T}^d$.

Then, $(\mathcal{Q}_{N})_{N}$ converges weakly to the probability measure $\mathcal{Q}$. \label{th:flu1}
\end{theorem}
The proof of this theorem is very close to the one presented for the Zero-Range process in \cite{K.L.} and for this reason we have omitted it.
We note that, since the proof of the Boltzmann-Gibbs in \cite{K.L.} relies on the ergodicity of the Markov process, we can use the ergodic
properties of the generator $\mathcal{L}_{S}$ to obtain the result. We also remark if one considers the process $\eta_{t}$ with generator
$\mathcal{L}_{P}$, we obtain in the limit the same Ornstein-Uhlenbeck process as described above. We only stress that, in this case, we derive
the Boltzmann-Gibbs Priciple in a similar way as we do in the proof of the One Block estimate via the Relative Entropy method.

Finally, in section \ref{gap}  we investigate the spectral gap for the process in finite volume for $d=1$. Our aim is to study the dependence on the
size of the system of the spectral gap of the process. This, as is very well known, scales as $1/N^2$ for SSEP
on all hyperplanes $\Sigma_{N,k}$, uniformly in $k$. For our models, due to presence of constraints, the uniformity in $k$ is certainly lost,
see remark \ref{remarkgap}\\
In order to illustrate our results we need to introduce a few additional notation. Fix an integer $N$ and denote by $\Lambda_N$ the box of size
$N$, $\Lambda_N=\{1,2\dots N\}$ and by $\chi^N$ the configuration space $\chi^N=\{0,1\}^N$. In order to define the generator on $\chi^N$ we
could use definition \ref{eq:lpgenerator} with the sum restricted to $x,y\in\Lambda_N$. However, some care should be put when defining the jump
rate for sites close to the boundary of $\Lambda_N$, since  $c(x,y,\eta)$ as defined in (\ref{rates}) depend not only on the configuration on
$x$ and $y$ but also on their neighbouring sites which can be outside $\Lambda_N$. We denote by $\partial\Lambda_N$ the {\sl boundary set}
including all sites which do not belong to $\Lambda_N$ and are nearest neighbour to at least one site in $\Lambda_N$,
$\partial\Lambda_N=\{x\in\mathbb Z: x\not\in\Lambda_N, d(x,\Lambda_N)=1\}$, where as usual the distance between a point $x$ and the set
$\Lambda_N$ is the infimum of the distances between $y\in\Lambda_N$ and $x$. A possible way to define the finite volume generator is to imagine
that the configuration in the boundary set is frozen to a reference configuration, $\sigma$, and to define the finite volume rates as
$c^{\sigma}(x,y,\eta)=c(x,y,\eta\cdot \sigma)$ where $c$ are the rates in (\ref{rates}) and
$\eta\cdot\sigma\in\{0,1\}^{|\Lambda_N|+|\partial\Lambda_N|}$ is the configuration equal to $\eta(x)$ on sites  $x\in\Lambda_N$ and to
$\sigma(x)$ on $x\in\partial\Lambda_N$. In the following we make the choice $\sigma(x)=0$ in $x\in\partial\Lambda_N$ and we denote by
${\mathcal{L}}_{P,\Lambda_N}$ and ${\mathcal{L}}_{\theta,\Lambda_N}^N$ the Markov processes with this choice corresponding to
(\ref{eq:lpgenerator}) and (\ref{eq:sumgenerator}). For sake of clarity, we explicitly write the second one




\begin{equation} \label{perturbedgeneratoronomegaN}
({\mathcal{L}}_{\theta,\Lambda_N}^Nf)(\eta)=\sum_{x\in{\Lambda_{N}\setminus\{1,N-1\}}}
\Big(c(x,x+1,\eta)+{N^{\theta-2}}\Big)\eta(x)(1-\eta(x+1))(f(\eta^{x,x+1})-f(\eta))
\end{equation}
\begin{equation*}
+\sum_{x\in{\Lambda_{N}\setminus\{2,N\}}} \Big(c(x,x-1,\eta)+{N^{\theta-2}}\Big)\eta(x)(1-\eta(x-1))(f(\eta^{x,x-1})-f(\eta))
\end{equation*}
\begin{equation*}
+\Big(\eta(3)+{N^{\theta-2}}\Big)\Big(\eta(1)-\eta(2)\Big)\Big(f(\eta^{1,2})-f(\eta)\Big)
\end{equation*}
\begin{equation*}
+\Big(\eta(N-2)+{N^{\theta-2}}\Big)\Big(\eta(N-1)-\eta(N)\Big)\Big(f(\eta^{N-1,N})-f(\eta)\Big)
\end{equation*}
where $c(x,y,\eta)$ and $\eta^{x,y}$ was defined in (\ref{rates}) and (\ref{etaxy}), respectively.

Let, with a slight abuse of notation, $\Sigma_{N,k}$ denote again the hyperplanes with  $k$ particles, namely those in (\ref{hyperplane}) but with the sum running over $\Lambda_N$. For each $k$, the Markov process generated by ${\mathcal{L}}_{\theta,\Lambda_N}^N$ is irreducible on $\Sigma_{N,k}$. The same holds for the process generated by ${\mathcal{L}}_{P,\Lambda_N}$ but only for $k>N/3$. This can be again proved by using the fact that the model is non cooperative with two particles at distance at most two being a mobile cluster. In both cases the unique invariant measure is the uniform measure, $\nu_{N,k}$.


For a generator $\mathcal L$ with invariant measure $\mu$,
denote by $\lambda_{N}({\mathcal L})$ its spectral gap, defined by
\begin{equation*}
\lambda_{N}({\mathcal L})=\inf_{f\in{L^2(\mu)}}\frac{\mathfrak{D}_{\mathcal L}(f,\mu)}{\textbf{Var}(f,\mu)},
\end{equation*}
where $\mathfrak{D}_{\mathcal L}(f,\mu)$ denotes the Dirichlet form defined by
\begin{equation} \label{dirichlet form}
\mathfrak{D}_{\mathcal L}(f,\mu)=\int_{\chi_{d}^N}-f(\eta) {\mathcal L} f(\eta)\mu(d\eta).
\end{equation}
In the following we will also use the shortened notation
$\mathfrak{D}_{P}(f,\mu)$ and $\mathfrak{D}_{\theta}(f,\mu)$
to denote the Dirichlet form with generator ${\mathcal{L}}_{P,\Lambda_N}$ and ${\mathcal{L}}_{\theta,\Lambda_N}^N$, respectively.
Let $\rho=k/N$. We obtain that:
\begin{proposition} \label{spectralgaplarge}
Fix $k>N/3$. For the Markov process with generator ${\mathcal{L}}_{P,\Lambda_N}$, there exists a constant $C$ that does not depend on $N$ nor $k$ such that
\begin{equation*}
\lambda_{N}({\mathcal{L}}_{\theta,\Lambda_N}^N)\geq\lambda_N({\mathcal{L}}_{P,\Lambda_N})\geq{C\frac{(\rho-1/3)}{\rho N^2}}.
\end{equation*}
\end{proposition}
\begin{proposition} \label{spectralgapsmall}
Fix $k\leq{N/3}$. For the Markov process with generator $\mathcal{L}_{\theta,\Lambda_N}^N$, where $\theta=1$, there exists a constant $C$ that
does not depend on $N$ nor $k$ such that:
\begin{equation*}
\lambda_{N}({\mathcal{L}}_{\theta,\Lambda_N}^N)\geq{C\frac{\rho^2}{ N^2}}.
\end{equation*}
\end{proposition}

\section{The Relative Entropy Method}
\label{relativeentropy}
In this section, we prove Theorem \ref{th:hlrem}. Let $\rho_{0}:\mathbb{T}^d\rightarrow{[0,1]}$ be a profile of class
$C^{2+\epsilon}(\mathbb{T}^d)$ that satisfies the bound condition (\ref{boundcondinitialprofile}) and let $(\mu^{N})_{N}$ be a sequence of
probability measures on $\chi_{d}^N$ that satisfies (\ref{entropymunurho0}).
Fix a time $t\geq{0}$. Denote by $S_{t}^{N,P}$ the semigroup associated to the generator $\mathcal{L}_{P}$ speeded up by $N^2$ and by
$\mu_{t}^{N}$ the distribution of the process at time $t$.
Denote by $\nu_{\rho(t,\cdot)}^{N}$ the product measure with slowly varying parameter associated to the profile $\rho(t,\cdot)$:
\begin{equation*}
\nu_{\rho(t,\cdot)}^{N}\Big\{\eta,\eta(x)=1\Big\}=\rho(t,x/N).
\end{equation*}
It is well known that in order to prove the Theorem \ref{th:hlrem} it is enough to show:

\begin{theorem}
Let $\rho_{0}:\mathbb{T}^d\rightarrow{\mathbb{R}}$ be an initial profile of class $C^{2+\epsilon}(\mathbb{T}^d)$ that satisfies the bound
condition (\ref{boundcondinitialprofile}) and $(\mu^{N})_{N}$ a sequence of probability measures in $\chi_d^N$ that satisfies the condition
(\ref{entropymunurho0}). Then, for every $t\geq{0}$
\begin{equation*}
H\Big(\mu^{N}_{t}/\nu_{\rho(t,\cdot)}^{N}\Big)=o(N^d),
\end{equation*}
where $\rho(t,u)$ is a smooth solution of (\ref{eq:porous m}).
\end{theorem}

\begin{proof}
Fix $\alpha\in(0,1)$ and an invariant measure $\nu_{\alpha}$. Let
\begin{equation*}
\psi_{t}^{N}=\frac{d\nu_{\rho(t,\cdot)}^{N}}{d\nu_{\alpha}} \\,\hspace{0,5cm} f_{N}(t)=\frac{d\mu_{t}^{N}}{d\nu_{\alpha}} \\,\hspace{0,5cm}
H_{N}(t)=H\Big(\mu^{N}_{t}/\nu_{\rho(t,\cdot)}^{N}\Big).
\end{equation*}
Since the measures $\nu_{\rho(t,\cdot)}^{N}$ and $\nu_{\alpha}$ are product, it is very simple to obtain an expression for $\psi_{t}^{N}$:
\begin{equation*}
\psi_{t}^{N}(\eta)=\frac{1}{Z_{t}^{N}}\exp\Big\{{\sum_{x\in{\mathbb{T}_{N}^d}}\eta(x)\lambda(t,x/N)}\Big\},
\end{equation*}
where
\begin{equation*}
\lambda(t,u)=\log \Big(\frac{\rho(t,u)(1-\alpha)}{\alpha(1-\rho(t,u))}\Big),
\end{equation*}
and $Z_{t}^{N}$ is a renormalizing constant.

In order to prove the result, we are going to show that
\begin{equation*}
H_{N}(t)\leq{o(N^d)+\frac{1}{\gamma}\int_{0}^{t}H_{N}(s)ds}
\end{equation*}
for some $\gamma>0$, and apply Gronwall inequality to conclude.

There is a celebrated estimate for the entropy production due to Yau \cite{Y.}:
\begin{equation}  \label{eq:entropyofYau}
\partial_{t}H_{N}(t)\leq{\int\Big\{\frac{N^{2}\mathcal{L}_{P}^{*}\psi_{t}^{N}(\eta)}{\psi_{t}^{N}(\eta)}-
\partial_{t}\log\psi_{t}^{N}(\eta)\Big\}f_{N}(t)(\eta)\nu_{\alpha}}(d\eta),
\end{equation}
where $\mathcal{L}_{P}^{*}$ is the adjoint operator of $\mathcal{L}_{P}$ in $L^{2}(\nu_{\alpha})$.

Here and after, for a local function $f$ we use the notation $\tilde{f}(\rho)=E_{\nu_{\rho}}[f(\eta)]$. By simple computations together with the
One-block estimate whose proof is presented at the end of this section, we can rewrite
$(\psi_{t}^{N}(\eta))^{-1}\{N^{2}\mathcal{L}_{P}\psi_{t}^{N}(\eta)-\partial_{t}\psi_{t}^{N}(\eta)\}$ as
\begin{equation*}
\sum_{x\in{\mathbb{T}_{N}^d}}\sum_{j=1}^{d}\partial^{2}_{u_{j}}\lambda(t,x/N)\Big\{\tilde{h}(\eta^{l}(x))-\tilde{h}(\rho(t,x/N))-\tilde{h}'(\rho(t,x/N))[\eta^{l}(x)-\rho(t,x/N)]\Big\}
\end{equation*}
\begin{equation*}
+\sum_{x\in{\mathbb{T}_{N}^d}}\sum_{j=1}^{d}(\partial_{u_{j}}\lambda(t,x/N))^{2}\Big\{\tilde{g}(\eta^{l}(x))-\tilde{g}(\rho(t,x/N))-\tilde{g}'(\rho(t,x/N))[\eta^{l}(x)-\rho(t,x/N)]\Big\},
\end{equation*}
plus a term of order $o(N^d)$, where
\begin{equation} \label{function h_{j}}
h_{j}(\eta)=\eta(0)\eta(e_{j})+\eta(0)\eta(-e_{j})-\eta(-e_{j})\eta(e_{j}),
\end{equation}
\begin{equation}  \label{function g_{j}}
g_{j}(\eta)=c(0,e_{j},\eta)(\eta(0)-\eta(e_{j}))^2
\end{equation}
and
\begin{equation} \label{empiricaldensity}
\eta^{l}(x)=\frac{1}{(2l+1)^d}\sum_{|y-x|\leq{l}}\eta(y).
\end{equation}

Repeating standard arguments of the relative entropy method, the result follows. We refer the reader to chapter 6 of \cite{K.L.} for details.
\end{proof}
\subsection{One-Block estimate}
\quad \vspace{0.5cm}

The main difficulty in the derivation of the One-Block estimate is the fact that, as already discussed, the process is not ergodic on
hyperplanes with a fixed particle number. In order to overcome this problem, we separate the set of configurations into two sets: the
irreducible component  that contains all configuration with at least one d-dimensional hypercube of particles of linear size $2$  (and all the
configurations that can be connected to them) and the remaining configurations. In the first case the standard proof is easily adapted, while
for the second case we will use as a key ingredient the fact that this set has small measure with respect to $\nu^{N}_{\rho(t,\cdot)}$.

Now we introduce some notation. Denote the also called Dirichlet form associated to the generator $\mathcal{L}_{P}$ and a measure $\mu$ in
${\chi_{d}^N}$, by $D_{P}(f,\mu)$ which is defined on positive functions by
\begin{equation} \label{dirichletformsqrtf}
D_{P}(f,\mu)=\mathfrak{D}_{P}(\sqrt{f},\mu)
\end{equation}
 and $\mathfrak{D}_{P}(f,\mu)$ was defined in (\ref{dirichlet form}). Let $f^{N,P}_{t}$ denote the
Radon-Nikodym density of $\mu^{N,P}(t)=\frac{1}{t}\int_{0}^{t}\mu^NS^{N,P}_{s}ds$ with respect to $\nu_{\alpha}$.

\begin{lemma}({One-block Estimate}) \label{th:oneblockmmp}

\quad\
 For every local function $\psi$ and for small $\gamma$
\begin{equation*}
\limsup_{l\rightarrow{+\infty}}\\\limsup_{N\rightarrow{+\infty}}\\
\\\int\frac{1}{N^d}\sum_{x\in{\mathbb{T}_{N}^d}}\tau_{x}V_{l,\psi}(\eta)f^{N,P}_{t}(\eta)\nu_{\alpha}(d\eta)\leq{\frac{1}{\gamma N^d}\int_{0}^{t}H_{N}(s)ds}
\end{equation*}
where
\begin{equation} \label{eq:Vreplacement}
V_{l,\psi}(\eta)=\Big|\frac{1}{(2l+1)^d}\sum_{|y|\leq{l}}\tau_{y}\psi(\eta)-\tilde{\psi}(\eta^{l}(0))\Big|
\end{equation}
and with $\eta^{l}(0)$ as defined in (\ref{empiricaldensity}).
\end{lemma}
\begin{proof}
Fix $x\in{\mathbb{T}_{N}^d}$ and denote by $\mathcal{Q}_{x,l}$ the set of configurations in the box of center $x$ and radius $l$ containing at least one d-dimensional hypercube of linear size $2$ which is completely filled:
\begin{equation*}
\mathcal{Q}_{x,l}=\Big\{\eta: \sum_{y\in C^{x}_y}\prod_{z \in Q_y }\eta(z)\geq{1}\Big\}.
\end{equation*}
where $Q_y=\{z: |z_i-y_i|\in\{0,1\}~ \forall i\in\{1,\dots,d\}\}$ and $C^{x}_y=\{y: |y-x|\leq{l},~Q_y\subset \mathbb T_N^d\}$. We denote by
$\mathcal{E}_{x,l}$ the irreducible set which contains $\mathcal{Q}_{x,l}$ (and all configurations that can be connected via an allowed path to
one in  $\mathcal{E}_{x,l}$) and we split the integral that appears in the statement of the Lemma into
\begin{equation} \label{oneblockintegralergodic}
\int\frac{1}{N^d}\sum_{x\in{\mathbb{T}_{N}^d}}\tau_{x}V_{l,\psi}(\eta)1_{\{\mathcal{E}_{x,l}\}}(\eta)f^{N,P}_{t}(\eta)\nu_{\alpha}(d\eta)
\end{equation}
\begin{equation} \label{oneblockintegralnotergodic}
+\int\frac{1}{N^d}\sum_{x\in{\mathbb{T}_{N}^d}}\tau_{x}V_{l,\psi}(\eta)1_{\{\mathcal{E}_{x,l}^{c}\}}(\eta)f^{N,P}_{t}(\eta)\nu_{\alpha}(d\eta).
\end{equation}
Note that the probability of the ergodic set $\mathcal{E}_{x,l}$ converges rapidly to one with $l$, indeed the following holds:
\begin{equation}
\label{goodproba}
\nu^N_{\rho(s,.)}(\mathcal{E}_{x,l})\geq 1-(1-\delta_{0}^{2^d})^{l^d}
\end{equation}
where we used  the fact that
the initial profile is bounded away from zero, namely $\rho(s,u)\geq{\delta_{0}}$  $\forall{u\in{\mathbb{T}^d}}$.
Since $H(\mu^{N}/\nu_{\alpha})=O(N^d)$ and the entropy decreases with time, it holds that $H(f^{N,P}_{t})=O(N^d)$. This
implies that the Dirichlet form of $f^{N,P}_{t}$ is bounded above by $CN^{d-2}$.
For the term (\ref{oneblockintegralergodic}), since we are restricting on an irreducible set, we can repeat the standard arguments of the One-block estimate and conclude:
\begin{equation*}
\limsup_{l\rightarrow{+\infty}}\\\limsup_{N\rightarrow{+\infty}}\\\\\sup_{D_{P}(f,\nu_{\alpha})\leq{CN^{d-2}}}
\\\int\frac{1}{N^d}\sum_{x\in{\mathbb{T}_{N}^d}}\tau_{x}V_{l,\psi}(\eta)1_{\{\mathcal{E}_{x,l}\}}(\eta)f(\eta)\nu_{\alpha}(d\eta)=0.
\end{equation*}
We now deal with the term (\ref{oneblockintegralnotergodic}) and, to keep notation simple, we drop the integral with respect to time. The
entropy inequality allows to bound it above by
\begin{equation*}
\frac{H\Big(\mu^{N}_{s}/\nu^N_{\rho(s,u)}\Big)}{\gamma N^d}+\frac{1}{\gamma N^d}\log \int \exp\Big\{\gamma
\sum_{x\in{\mathbb{T}_{N}^d}}\tau_{x}V_{l,\psi}(\eta)1_{\{\mathcal{E}_{x,l}^c\}}(\eta)\Big\}\nu^N_{\rho(s,.)}(d\eta),
\end{equation*}
for every $\gamma>0$.

The term on the left hand side is $H_{N}(s)/\gamma N^d$. On the other hand, since $V_{l,\psi}$ is bounded and by H\"{o}lder inequality, the
second term of last expression can be bounded by
\begin{equation*}
\frac{1}{N\gamma (2l+1)^d}\sum_{x\in{\mathbb{T}^d_{N}}}\log \int \exp\Big\{\gamma (2l+1)^d
C1_{\{\mathcal{E}_{x,l}^c\}}(\eta)\Big\}\nu^N_{\rho(s,.)}(d\eta)
\end{equation*}
\begin{equation*}
=\frac{1}{N\gamma (2l+1)^d}\sum_{x\in{\mathbb{T}^d_{N}}}\log \Big(\nu^N_{\rho(s,.)}(\mathcal{E}_{x,l}^c)(\exp\{\gamma(2l+1)^d C\}-1)+1\Big).
\end{equation*}

By using the upper bound on $\nu^N_{\rho(s,.)}(\mathcal{E}_{x,l}^c)$ which follows from (\ref{goodproba}) and the inequality $\log(x+1)\leq{x}$
which holds true $\forall{x}$, we can bound from above last expression  by
\begin{equation*}
\frac{1}{\gamma (2l+1)^d}(\exp\{\gamma(2l+1)^dC\}-1)(1-\delta_{0}^{2^d})^{l^d},
\end{equation*}
which vanishes as $l\rightarrow{+\infty}$ provided $2^d\gamma+\log(1-\delta_0^{2^d})<0$.
\end{proof}
\section{The Entropy Method}
\label{entropy}
Now, we prove Theorem \ref{th:hlrm}. The strategy of the proof is the same as given for the Zero-Range process in \cite{K.L.}. The main step is
the derivation of the One-Block and the Two-Blocks estimate, which are presented in the fifth section.

Fix a time $T>0$. Let $\mathcal{M}_{+}$ be the space of finite positive measures on $\mathbb{T}^d$ endowed with the weak topology. Consider a
sequence of probability measures $(Q_{N})_{N}$ on $D([0,T],\mathcal{M}_{+})$ corresponding to the Markov process $\pi_{t}^{N}$ as defined in
(\ref{eq:empirical}), starting from $\mu^{N}$.

First we prove that $(Q_{N})_{N}$ is a tight sequence. Then we prove the uniqueness of a limit point, by showing that the limit points of
$(Q_{N})_{N}$ are concentrated on trajectories of measures absolutely continuous with respect to the Lebesgue measure, equal to $\rho_0(u)du$ at
the initial time and whose density is concentrated on weak solutions of the hydrodynamic equation (\ref{eq:porous m}). By the uniqueness of
these solutions we conclude that $\pi_t^N$ has a unique limit point, concentrated on the trajectory with density $\rho(t,u)$, where $\rho(t,u)$
is the weak solution of equation (\ref{eq:porous m}).

We divide the proof in several steps, to make the exposition clearer. Fix a smooth function $H:\mathbb{T}^d\rightarrow{\mathbb{R}}$. Recall the
definition of the empirical measure in (\ref{eq:empirical}) and let
\begin{equation*}
<\pi_{t}^{N},H>=\frac{1}{N^{d}}\sum_{x\in{\mathbb{T}_{N}^d}}H\Big(\frac{x}{N}\Big)\eta_{t}(x).
\end{equation*}
By lemma A1.5.1 of \cite{K.L.}
\begin{equation*}
M^{N,H}_{t}=<\pi_{t}^{N},H>-<\pi_{0}^{N},H>-\int^{t}_{0}N^{2}\mathcal{L}_{\theta}^N<\pi_{s}^{N},H>ds
\end{equation*}
is a martingale with respect to the filtration $\mathcal{F}_{t}=\sigma(\eta_{s}, s\leq{t})$, whose quadratic variation is given by
\begin{equation} \label{quadraticvariation}
\int^{t}_{0}\frac{1}{N^{2d}}\sum_{x\in{\mathbb{T}_{N}^d}}\sum_{j=1}^{d}\Big(\partial^{N}_{u_{j}}H\Big(\frac{x}{N}\Big)\Big)^{2}\tau_{x}g^N_{j}(\eta_{s})ds,
\end{equation}
where
\begin{equation*}
\partial^{N}_{u_{j}}H\Big(\frac{x}{N}\Big)=N\Big[H\Big(\frac{x+e_{j}}{N}\Big)-H\Big(\frac{x}{N}\Big)\Big],
\end{equation*}
$g^N_{j}(\eta)=g_{j}(\eta)+(\eta(0)-\eta(e_{j}))^{2}N^{\theta-2}$ and $g_{j}(\eta)$ as defined in (\ref{function g_{j}}).

By elementary computations we can rewrite the integral part of the martingale as
\begin{equation} \label{integralpartofmartingale}
\int^{t}_{0}\frac{1}{N^d}\sum_{x\in{\mathbb{T}_{N}^d}}\sum_{j=1}^{d}\partial^{2}_{u_{j}}H\Big(\frac{x}{N}\Big)\tau_{x}h^N_{j}(\eta_{s})ds,
\end{equation}
where
\begin{equation*}
\partial^2_{u_{j}}H\Big(\frac{x}{N}\Big)=N^{2}\Big[H\Big(\frac{x+e_{j}}{N}\Big)+H\Big(\frac{x-e_{j}}{N}\Big)-2H\Big(\frac{x}{N}\Big)\Big]
\end{equation*}
and $h^N_{j}(\eta)=h_{j}(\eta)+N^{\theta-2}\eta(0)$ and $h_{j}(\eta)$ as defined in (\ref{function h_{j}}).
\subsection{Relative Compactness}

\quad\

 Following the same arguments as for the Zero-range in \cite{K.L.} it is easy to show that $(Q_{N})_{N}$ is tight.

\subsection{Uniqueness of Limit Points}
\quad \vspace{0.1cm}

At first we note that all limit points $Q$ of $(Q_{N})_{N}$ are concentrated on absolutely continuous measures with respect to the Lebesgue
measure since there is at most one particle per site. The limit points are equal to $\rho_{0}(u)du$ at the initial time by the hypothesis of
$(\mu^{N})_N$ being associated to the profile $\rho_{0}$. But, in order to prove that the limit points are concentrated on weak solutions of
equation (\ref{eq:porous m}), we need to write the integral part of the martingale (\ref{integralpartofmartingale}) as a function of the
empirical measure. This is the main difficulty in the proof of an hydrodynamical limit for a gradient system and we state it as a lemma:
\begin{lemma}{(Replacement Lemma)} \label{replacement lemma}
\\
For every $\delta>{0}$ and every local function $\psi$
\begin{equation*}
\limsup_{\epsilon\rightarrow{0}}\limsup_{N\rightarrow{+\infty}}\mathbb{P}_{\mu^{N}}
\Big[\int_{0}^{T}\frac{1}{N^d}\sum_{x\in{\mathbb{T}_{N}^d}}\tau_{x}V_{\epsilon N,\psi}(\eta_{s})ds\geq{\delta}\Big]=0,
\end{equation*}
where $V_{l,\psi}$ was defined in (\ref{eq:Vreplacement}).
\end{lemma}
We postpone the proof to the following section in order to make the exposition more clear. To keep notation simple, in the previous statement
and hereafter we write $\epsilon N$ for $[\epsilon N]$, it's integer part.

Lemma \ref{replacement lemma} means that
\begin{equation*}
\int_{0}^{t}\frac{1}{N^d}\sum_{x\in{\mathbb{T}_{N}^d}}\sum_{j=1}^{d}H\Big(s,\frac{x}{N}\Big)\Big\{\tau_{x}h_{j}(\eta_{s})-\tilde{h}(\eta_{s}^{\epsilon
N}(x))\Big\}ds
\end{equation*}
vanishes in probability as $N\rightarrow{+\infty}$ and then as $\epsilon\rightarrow{0}$, for every continuous function $H$.
This is the replacement that permits to close the integral part of the martingale in terms of the empirical measure. For details we refer the reader to
chapter 5 of \cite{K.L.}.
\section{Replacement Lemma}
\label{replacement}
This section is devoted to the proof of Lemma \ref{replacement lemma}. Before introducing the proof we give some fundamental inequalities that
are used in the sequel.

Fix a density $\alpha$ and an invariant measure $\nu_{\alpha}$, as reference. Let $(\mu^{N})_{N}$ be a sequence of probability measures on
$\chi_{d}^N$ and denote by $S_{t}^{N}$ the semigroup associated to the generator $\mathcal{L}_{\theta}^N$. Denote by $f_{t}^{N}$ the density of
$\mu^{N}S_{t}^{N}$ with respect to $\nu_{\alpha}$. As the process is evolving on a torus we have that
\begin{equation*}
H(\mu^{N}/\nu_{\alpha})\leq{C(\alpha)N^d},
\end{equation*}
and since the entropy decreases with time we obtain that $H(\mu^{N}S^{N}_{t}/\nu_{\alpha}^{N})\leq{C(\alpha)N^d}$, see Proposition A1.9.2 of
\cite{K.L.}. Denote by $(\mathcal{L}_{\theta}^N)^{*}$ the adjoint operator of $\mathcal{L}_{\theta}^N$ in $L^{2}(\nu_{\alpha})$. Since the
process is reversible with respect to $\nu_{\alpha}$, $f_{t}^{N}$ satisfies the following equality:
\begin{equation*}
\partial_{t}f_{t}^{N}=N^2(\mathcal{L}_{\theta}^N)^{*}f_{t}^{N}=N^2\mathcal{L}_{\theta}^Nf_{t}^{N}
\end{equation*}
with initial condition $f_{0}^{N}=\frac{d\mu^{N}}{d\nu_{\alpha}^{N}}$.

For $f:\chi_{d}^N\rightarrow{\mathbb{R}_{+}}$, recall the definition of the Dirichlet form
$D(f,\nu_{\alpha})=\mathfrak{D}(\sqrt{f},\nu_{\alpha})$ and $\mathfrak{D}(f,\nu_{\alpha})$ defined in (\ref{dirichlet form}). Since there exists
a constant $C$, such that $H_{N}(f_{0}^{N})\leq{CN^d}$, it can be proved that for all $t>{0}$ (see \cite{K.L.}, section 5.2), that
\begin{equation} \label{ergodic estimates}
H_{N}\Big(\frac{1}{t}\int_{0}^{t}f_{s}^{N}ds\Big)\leq{CN^d}, \\ \quad\
D_{\theta}\Big(\frac{1}{t}\int_{0}^{t}f_{s}^{N}ds,\nu_{\alpha}\Big)\leq{\frac{CN^{d-2}}{2t}}.
\end{equation}

The proof of the Replacement Lemma relies on the well-known One-block and Two-blocks estimates. At first we reduce the dynamical problem, since
the function depends on a trajectory, to a static one using the estimates obtained above. For that, we need to introduce some notation.

Let $\mu^{N}(T)$ be the Cesaro mean of $\mu^{N}S_{t}^{N}$, namely:
\begin{equation*}
\mu^{N}(T)=\frac{1}{T}\int_{0}^{T}\mu^{N}S_{t}^{N}dt
\end{equation*}
and $\bar{f}_{T}^{N}$ the Radon-Nikodym density of $\mu^{N}(T)$ with respect to $\nu_{\alpha}$. Recall that in the beginning of this section we
have obtained estimates for the entropy and the Dirichlet form of $\bar{f}_{T}^{N}$. As a consequence, to prove the Replacement Lemma it is
enough to show that
\begin{equation} \label{eq:replacement}
\limsup_{\epsilon\rightarrow{0}}\limsup_{N\rightarrow{+\infty}}\sup_{D_{\theta}(f,\nu_{\alpha})\leq{CN^{d-2}}}\int
\frac{1}{N^d}\sum_{x\in{\mathbb{T}_{N}^d}}\tau_{x}V_{\epsilon N,\psi}(\eta)f(\eta)\nu_{\alpha}(d\eta)=0,
\end{equation}
for every $C<{\infty}$. This is a consequence of the following two results:
\begin{lemma}({One-block Estimate}) \label{th:oneblock}
\\
For every finite constant $C$,
\begin{equation*}
\limsup_{l\rightarrow{+\infty}}\\\limsup_{N\rightarrow{+\infty}}\\\sup_{D_{\theta}(f,\nu_{\alpha})\leq{CN^{d-2}}}
\\\int\frac{1}{N^d}\sum_{x\in{\mathbb{T}_{N}^d}}\tau_{x}V_{l,\psi}(\eta)f(\eta)\nu_{\alpha}(d\eta)=0,
\end{equation*}
where $V_{l,\psi}$ was defined in (\ref{eq:Vreplacement}).
\end{lemma}

\begin{lemma}({Two-blocks Estimate}) \label{th:twoblocks}
\\
For every finite constant $C$,
\begin{equation*}
\limsup_{l\rightarrow{+\infty}}\\\limsup_{\epsilon\rightarrow{0}}\\\limsup_{N\rightarrow{+\infty}}
\\\sup_{D_{\theta}(f,\nu_{\alpha})\leq{CN^{d-2}}}\\\sup_{|y|\leq{\epsilon N}}
\end{equation*}
\begin{equation} \label{eq:twoblocks}
\int\frac{1}{N^d}\sum_{x\in{\mathbb{T}_{N}^d}}\Big|\eta^{l}(x+y)-\eta^{\epsilon N}(x)\Big|f(\eta)\nu_{\alpha}(d\eta)=0.
\end{equation}
\end{lemma}
\subsection{One-Block Estimate}

\quad \vspace{0.1cm}

Since in this case we are considering the perturbed process $\mathcal{L}_{\theta}^N$, we can use the One-Block estimate for the SSEP speeded up
by $N^\theta$, which is enough to conclude the One-Block estimate for the perturbed process, the idea to proceed is the following. By the
definition of the Dirichlet form (\ref{dirichlet form}) and the generator $\mathcal{L}_{\theta}^N$:
\begin{equation}
D_{\theta}(f,\nu_{\alpha})=D_{P}(f,\nu_{\alpha})+N^{\theta-2}D_{S}(f,\nu_{\alpha}),\label{eq:dirisum}
\end{equation}
where $D_{P}$, (resp. $D_{S}$) denotes the Dirichlet form as defined in (\ref{dirichletformsqrtf}). Since the Dirichlet form is always positive,
and we are restricted to densities $f$ for which $D_{\theta}(f,\nu_{\alpha})\leq{CN^{d-2}}$, it holds that:
\begin{equation*}
D_{S}(f,\nu_{\alpha})\leq{N^{2-\theta} D_{\theta}(f,\nu_{\alpha})}\leq{{C N^{d-\theta}}}
\end{equation*}
Following the same arguments as for the Zero-Range in section $5$ of \cite{K.L.}, it is not hard to show that the modified One-Block estimate
\begin{equation*}
\limsup_{l\rightarrow{+\infty}}\\\limsup_{N\rightarrow{+\infty}}\\\sup_{D_{S}(f,\nu_{\alpha})\leq{CN^{d-\theta}}}
\\\int\frac{1}{N^d}\sum_{x\in{\mathbb{T}_{N}^d}}\tau_{x}V_{l,\psi}(\eta)f(\eta)\nu_{\alpha}(d\eta)=0.
\end{equation*}
holds for the SSEP. The main difference between the proofs, comes from the bounds on the Dirichlet forms. Having the bound
$D_{S}(f,\nu_{\alpha})\leq{CN^{d-2}}$ it provides the estimate $D^{l}_{S}(f_{l},\nu_{\alpha})\leq{C/N^{2}}$, where $f_{l}$ is the conditional
expectation of $f$ with respect to $\sigma$-algebra generated by $\{\eta(x),|x|\leq{l}\}$, while the bound
$D_{S}(f,\nu_{\alpha})\leq{CN^{d-\theta}}$, provides the estimate $D^{l}_{S}(f_{l},\nu_{\alpha})\leq{C/N^{\theta}}$, which is enough to conclude
the standard proof of the One-block estimate, since $\theta>0$.
\subsection{Two-Blocks Estimate}
\quad \vspace{0.1cm}

In order to show (\ref{eq:twoblocks}) it is enough to prove that
\begin{equation*}
\limsup_{l\rightarrow{+\infty}}\\\limsup_{\epsilon\rightarrow{+\infty}}\\\limsup_{N\rightarrow{+\infty}}
\\\sup_{D_{\theta}(f,\nu_{\alpha})\leq{CN^{d-2}}}\\\sup_{2l\leq{|y|\leq{\epsilon N}}}
\end{equation*}
\begin{equation*}
\int\frac{1}{N^d}\sum_{x\in{\mathbb{T}_{N}^d}}|\eta^{l}(x)-\eta^{l}(x+y)|f(\eta)\nu_{\alpha}(d\eta)=0.
\end{equation*}
We can write this integral as
\begin{equation*}
\int|\eta^{l}(0)-\eta^{l}(y)|\bar{f}(\eta)\nu_{\alpha}(d\eta),
\end{equation*}
where $\bar{f}$ denotes the average of all space translations of $f$.

In this case, we are able to separate the integral over configurations with small density of particles from the ones with high density. The
first case is easily treated. For remaining, since we are taking configurations with high density of particles, we are able to construct a
coupled process, in which each marginal evolves according to the SSEP and they are connected by jumps of $\mathcal{L}_{P}$. The main difference
between the proof for the Zero-Range Process in \cite{K.L.} comes from the estimate in the Dirichlet form of this process, which in this case we
are able to show that is bounded by $C(d,l)\epsilon^\theta$ and is enough to conclude. We divide the proof in several steps, in order to make
the exposition clearer.

\subsubsection{\textbf{Cut off of small densities}}
\quad \vspace{0.1cm}

At first we note that if we restrict last integral to the small density set
\begin{equation*}
\Omega_{0,y}=\{\eta:\sum_{|x|\leq{l}}\eta(x)\leq{2^d}\}\cap{\{\eta:\sum_{|x-y|\leq{l}}\eta(x)\leq{2^d}\}},
\end{equation*}
it is bounded above by $\frac{C(d)}{(2l+1)^d}$, which vanishes as $l\rightarrow{+\infty}$. So, in fact we just have to consider the integral
over the complementary set $\Omega^{c}_{0,y}$.
\subsubsection{\textbf{Reduction to microscopic cubes}}
\quad \vspace{0.1cm}

Here we need to introduce some notation. Fix a positive integer $l$ and a site $x$, denote by $\wedge_{l}(x)$ the box centered at $x$ with
radius $l$. Let $\chi^{2,l}$ denote the configuration space $\{0,1\}^{\wedge_{l}(0)}\times{\{0,1\}^{\wedge_{l}(0)}}$, $\xi=(\xi_{1},\xi_{2})$
the configurations of $\chi^{2,l}$ and by $\nu_{\alpha}^{2,l}$ the product measure $\nu_{\alpha}$ restricted to $\chi^{2,l}$. Denote by
$f_{y,l}$ the conditional expectation of $f$ with respect to the $\sigma$-algebra generated by $\{\eta(z),z\in{\wedge_{l}(0,y)}\}$, where
\begin{equation*}
\wedge_{l}(0,y)=\wedge_{l}(0)\cup{\wedge_{l}(y)}.
\end{equation*}
Since, $\eta^{l}(0)$ and $\eta^{l}(y)$ depend on $\eta(x)$, for $x\in{\wedge_{l}(0,y)}$, we are able to replace $\bar{f}$ by $\bar{f}_{y,l}$ and
rewrite our desired limit as
\begin{equation*}
\limsup_{l\rightarrow{+\infty}}\\\limsup_{\epsilon\rightarrow{0}}\\\limsup_{N\rightarrow{+\infty}}
\\\sup_{D_{\theta}(f,\nu_{\alpha})\leq{CN^{d-2}}}\\\sup_{2l<|y|\leq{2\epsilon N}}
\end{equation*}
\begin{equation*}
\int_{\Omega^{c}_{0,y}}|\xi_{1}^{l}(0)-\xi_{2}^{l}(0)|\bar{f}_{y,l}(\xi)\nu_{\alpha}^{2,l}(d\xi)=0.
\end{equation*}
\subsubsection{\textbf{Estimates on the Dirichlet form}}
\quad \vspace{0.1cm}

At first we need to introduce some notation since we are working in the space $\Omega^{c}_{0,y}$. Recall from (\ref{dirichlet form}) the
definition of the Dirichlet form of $f$. Since the Bernoulli product measures are homogeneous, we can define the Dirichlet form of $f$ as:
\begin{equation*}
D_{\theta}(f,\nu_{\alpha})=\sum_{\substack{x,z\in{\mathbb{T}_{N}^{d}}\\|x-z|=1}}I_{x,z}(f,\nu_{\alpha}),
\end{equation*}
where
\begin{equation} \label{relation I}
I^{\theta}_{x,z}(f,\nu_{\alpha})=I_{x,z}^{P}(f,\nu_{\alpha})+N^{\theta-2}I_{x,z}^{S}(f,\nu_{\alpha}),
\end{equation}
\begin{equation*}
I_{x,z}^{P}(f,\nu_{\alpha})=\int_{\Omega^{c}_{0,y}}c(x,z,\eta)\Big[\sqrt{f(\eta^{x,z})}-\sqrt{f(\eta)}\Big]^{2}\nu_{\alpha}(d\eta)
\end{equation*}
and
\begin{equation*}
I_{x,z}^{S}(f)=\int_{\Omega^{c}_{0,y}}\Big[\sqrt{f(\eta^{x,z})}-\sqrt{f(\eta)}\Big]^{2}\nu_{\alpha}(d\eta).
\end{equation*}

By the definition of $\bar{f}$ and since the Dirichlet form is convex, this implies that
$D_{\theta}(\bar{f},\nu_{\alpha})\leq{D_{\theta}(f,\nu_{\alpha})}$.

The main step in the proof, consists in obtaining an upper bound for the Dirichlet form of $\bar{f}_{y,l}$ from the upper bound of the Dirichlet
form of $\bar{f}$, in such a way that we can use the ergodic properties of the Markov process. The idea is to obtain a limit density with
Dirichlet form in $\wedge_{l}(0,y)$ equal to $0$, and then by using the irreducibility we can decompose that density along the hyperplanes. As
the boxes $\wedge_{l}(0)$ and $\wedge_{l}(y)$ have no communication, we cannot define the Dirichlet form in $\wedge_{l}(0,y)$ as the sum of the
Dirichlet forms in $\wedge_{l}(0)$ and in $\wedge_{l}(y)$, we must add a term that connects both boxes. Define on positive densities
$f:\chi^{2,l}\rightarrow{\mathbb{R}_{+}}$:
\begin{equation}
D^{2,l}_{\theta}(f,\nu_{\alpha}^{2,l})=I^{l}_{0,0}(f,\nu_{\alpha}^{2,l})+\sum_{\substack{x,z\in{\wedge_{l}(0)}\\|x-z|=1}}I^{1,l}_{x,z}(f,\nu_{\alpha}^{2,l})+
\sum_{\substack{x,z\in{\wedge_{l}(0)}\\|x-z|=1}}I^{2,l}_{x,z}(f,\nu_{\alpha}^{2,l}), \label{dirichletformoff_{0,y}}
\end{equation}
where
\begin{equation*}
I^{1,l}_{x,z}(f,\nu_{\alpha}^{2,l})=\int_{\Omega^{c}_{0,y}}\Big[\sqrt{f(\xi_{1}^{x,z},\xi_{2})}-\sqrt{f(\xi_{1},\xi_{2})}\Big]^{2}\nu_{\alpha}^{2,l}(d\xi_{1},d\xi_{2}),
\end{equation*}
\begin{equation*}
I^{2,l}_{x,z}(f,\nu_{\alpha}^{2,l})=\int_{\Omega^{c}_{0,y}}\Big[\sqrt{f(\xi_{1},\xi_{2}^{x,z})}-
\sqrt{f(\xi_{1},\xi_{2})}\Big]^{2}\nu_{\alpha}^{2,l}(d\xi_{1},d\xi_{2}).
\end{equation*}
We still have to define the term that connects both boxes, namely $I^{l}_{0,0}(f,\nu_{\alpha}^{2,l})$. Notice that
$\Omega^{c}_{0,y}=\Omega_{0}^c\cup{\Omega_{y}^c}$, where
\begin{equation*}
\Omega_{0}=\{\eta\in{\{0,1\}^{\wedge_{l}(0)}}:\sum_{x\in{\wedge_{l}(0)}}\eta(x)<{2^d+1}\}
\end{equation*}
 and
\begin{equation*}
\Omega_{y}=\{\eta\in{\{0,1\}^{\wedge_{l}(y)}}:\sum_{x\in{\wedge_{l}(y)}}\eta(x)<{2^d+1}\}.
\end{equation*}
 Then we define
\begin{equation*}
I^{l}_{0,0}(f,\nu_{\alpha}^{2,l})=\int_{\Omega_{0}^c}
\sum_{j=1}^{d}c(0,e_{j},\xi_{1})\Big[\sqrt{f(\xi_{1}^{0,-},\xi_{2}^{0,+})}-\sqrt{f(\xi)}\Big]^{2}\nu_{\alpha}^{2,l}(d\xi_{1},d\xi_{2})
\end{equation*}
\begin{equation*}
+\int_{\Omega_{y}^c}\sum_{j=1}^{d}c(0,e_{j},\xi_{2})\Big[\sqrt{f(\xi_{1}^{0,+},\xi_{2}^{0,-})}-\sqrt{f(\xi)}\Big]^{2}\nu_{\alpha}^{2,l}(d\xi_{1},d\xi_{2}),
\end{equation*}
where $\xi_{i}^{0,\pm}=\xi_{i}\pm\partial_{0}$ and $\partial_{0}$ is the configuration with one particle at site ${0}$ and in the rest empty.

This Dirichlet form corresponds to a particle system on $\wedge_{l}(0)\times{\wedge_{l}(0)}$, where the marginal processes evolve as SSEP and
where particles can jump from the origin of one marginal process to the origin of the other and vice-versa, according to the jumps of the
generator $\mathcal{L}_{P}$. Using Schwarz inequality and the definition of $\bar{f}_{y,l}$, we have that:
\[
\begin{array}{cc}
I_{x,z}^{1,l}(\bar{f}_{y,l},\nu_{\alpha}^{2,l})\leq{I_{x,z}^{S}(\bar{f},\nu_{\alpha}^{2,l})} \hspace{0.5cm}and &
I_{x,z}^{2,l}(\bar{f}_{y,l},\nu_{\alpha}^{2,l})\leq{I_{x,z}^{S}(\bar{f},\nu_{\alpha}^{2,l})}
\end{array}
\]
Then, by equality (\ref{relation I}) and the estimate on the Dirichlet form of $f$, it holds that
$D_{S}(\bar{f},\nu_{\alpha})\leq{CN^{d-\theta}}$ which together with $l/N<\epsilon$, implies that:
\begin{equation*}
\sum_{\substack{x,z\in{\wedge_{l}(0)}\\|x-z|=1}}I^{1,l}_{x,z}(f,\nu_{\alpha}^{2,l})+
\sum_{\substack{x,z\in{\wedge_{l}(0)}\\|x-z|=1}}I^{2,l}_{x,z}(f,\nu_{\alpha}^{2,l})\leq{\frac{2(2l+1)^{d-1}2lD_{S}(\bar{f},\nu_{\alpha})}{N^d}}
\end{equation*}
\begin{equation} \label{firstpartofdirichletform}
\leq{C(d,l)N^{-\theta}}
\end{equation}

In last expression $C(d,l)$ is a constant that depends on $d$ and $l$. In what follows it may vary from line to line.

 It remains to obtain an upper bound for $I^{l}_{0,0}(\bar{f}_{y,l})$. By using Schwarz
inequality and the definition of $\bar{f}_{y,l}$, we can bound $I^{l}_{0,0}(\bar{f}_{y,l},\nu_{\alpha}^{2,l})$ by
\begin{equation*}
\int_{\Omega_{0}^c}\sum_{j=1}^{d}c(0,e_{j},\eta)\Big[\sqrt{\bar{f}(\eta^{0,y})}-\sqrt{\bar{f}(\eta)}\Big]^{2}\nu_{\alpha}(d\eta)
\end{equation*}
\begin{equation*}
+\int_{\Omega_{y}^c}\sum_{j=1}^{d}c(y,y+e_{j},\eta)\Big[\sqrt{\bar{f}(\eta^{0,y})}-\sqrt{\bar{f}(\eta)}\Big]^{2}\nu_{\alpha}(d\eta).
\end{equation*}
Let $E^{l,j}_{0,y}(\bar{f},\nu_{\alpha})$ $j=1,2$ denote the first (respectively second) expectation above. In order to keep the proof clear we
are going to estimate $E^{l,1}_{0,y}(f,\nu_{\alpha})$ and state it as a lemma. We note that a similar argument provides the same bound for
$E^{l,2}_{0,y}(f,\nu_{\alpha})$.

\begin{lemma} \label{pathlemma}
Let $f$ be a density such that $D_{\theta}(f,\nu_{\alpha})\leq{CN^{d-2}}$ and let $y\in{\mathbb{T}^d_{N}}:2l<|y|\leq{2\epsilon N}$. Then,
\begin{equation*}
E^{l,1}_{0,y}(f,\nu_{\alpha})\leq{C(d,l)\epsilon^{\theta}}.
\end{equation*}
\end{lemma}

\begin{proof}
The idea consists in expressing the exchange $\eta^{0,y}$ by means of allowed
nearest
neighbor exchanges $\eta^{x,x+1}$  of the generator $\mathcal{L}_{\theta}^N$.
Since the integral is restricted to $\Omega_{0}^c$ we have for certain $2^d+1$ particles in $\wedge_{l}(0)$. We discuss for definiteness the
path in the case $\eta(0)=1$, $\eta(y)=0$, the other possibility can be treated analogously. In order to bring the particle from  $0$ to $y$ we
first move  $2^d$ of the particles in $\wedge_{l}(0)$
close to $0$ by means of the jumps that corresponds to $\mathcal{L}_{S}$. Then we arrange them in order to form a d-dimensional hypercube of linear size $2$, which is a mobile cluster. Now we can shift this cluster plus the particle originally in $0$ in each of the $d$ directions by using only the jumps in $\mathcal L_P$ and we bring them close to $y$.  In figure \ref{path1} and \ref{path2} we show as an example the path which allows to shift in the $e_1$ and $e_2$ direction the mobile $2\times 2$ square of particles (black circles) plus the particle originally in $0$ (grey circle).

\begin{figure}[htp]
  \includegraphics[width=13cm]{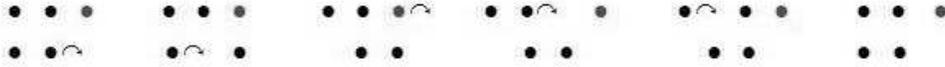}\\
   \caption{Moving in the direction $e_{1}$}\label{path1}
\end{figure}

\begin{figure}[htp]
  \includegraphics[width=12cm]{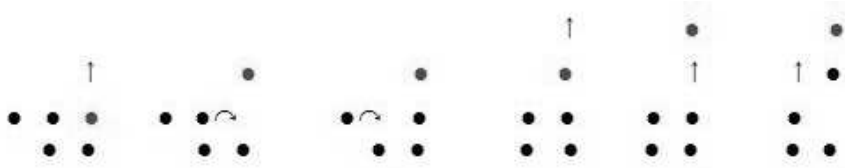}\\
  \caption{Moving in the direction $e_{2}$}\label{path2}
\end{figure}

Then we drop particle  that was originally in $0$ in site $y$
 and bring back the mobile cluster alone to $\wedge_{l}(0)$ again using moves in $\mathcal L_P$. Finally, using the jumps of $\mathcal{L}_{S}$, we put them in their
initial positions.


Note that the above path is uniquely defined only after choosing  the $2^d$ particles in the box $\wedge_{l}(0)$ which are used to form the mobile cluster and whose existence in guaranteed by the fact that the integral is restricted to $\Omega_{0}^c$.
This will bring an entropy factor which corresponds to the number of possible initial positions for the $2^d$ particles, which is bounded by
$l^{d2^d}$.

Once for a given $\eta$ the initial position of the $2^d$ particles which will be used to form the mobile cluster is fixed, namely
$x_{1},..,x_{2^d}$, we let $\tau_{\gamma^{0,y}}\tau_{\gamma^{0,y}-1}...\tau_{1}(\eta)$ be the sequence of nearest neighboring exchanges that
represents the above described path. We can therefore rewrite, for any function $g$,
the square difference
$\left(g(\eta^{0,y})-\bar{g}(\eta)\right)^2$ as a telescopic sum:
\begin{equation*}
\Big[g(\eta^{0,y})-g(\eta)\Big]^{2}=
\Big[\sum_{k=1}^{\gamma^{0,y}}g\Big(\tau_{k}\prod_{i=1}^{k-1}\tau_{i}(\eta)\Big)-g\Big(\prod_{i=1}^{k-1}\tau_{i}(\eta)\Big)\Big]^{2}
\end{equation*}
where $\gamma^{0,y}$ denotes the number of steps of the path which is less than $C(d)(l+\epsilon N)$.

Now we separate the part of the path in which the moves are performed via jumps in $\mathcal{L}_S$ from the one that can be performed by using
$\mathcal{L}_P$ and,
 by using the elementary inequality
$(x+y)^2\leq{2x^2+2y^2}$, we bound last expression by
\begin{equation*}
\Big[\sum_{k=1}^{C(d)l}g\Big(\tau_{k}\prod_{i=1}^{k-1}\tau_{i}(\eta)\Big)-g\Big(\prod_{i=1}^{k-1}\tau_{i}(\eta)\Big)\Big]^{2}
+\Big[\sum_{k=C(d)l+1}^{C(d)N\epsilon }g\Big(\tau_{k}\prod_{i=1}^{k-1}\tau_{i}(\eta)\Big)-
g\Big(\prod_{i=1}^{k-1}\tau_{i}(\eta)\Big)\Big]^{2}
\end{equation*}
Note that in the previous expression we had to take into account the number of steps in the paths, which is order $l$ for the terms involving
moves of $\mathcal L_S$ and of order $N\epsilon$ for those using $\mathcal{L}_{P}$, thanks to our choice of the path. Also note that the term on
the left hand side of last expression, refers to all the jumps inside the box of size $l$: when bringing the $2^d$ particles together and when
putting them in their initial position.

We now apply the above inequality with the choice $g=\sqrt{\bar f}$ and use again the Cauchy-Schwarz inequality on each single term and the fact that the rates of $\mathcal L_P$ (\ref{rates}) are bounded from below by $1$ thanks to our choice of the path for all the terms in the
second telescopic sum. This leads to the bound
\begin{equation*}
E^{l,1}_{0,y}(f,\nu_{\alpha})\leq C(d)l\sum_{\substack{a_{1}\in{\wedge_{l}(0)}\\{..}\\{a_{2^d}\in{\wedge_{l}(0)}}}}\int_{\Omega_{0}^c}1_{\{a_{1}=x_{1},..,a_{2^d}=x_{2^d}\}}
\sum_{e_{i}}\Big[\sqrt{\bar{f}(\eta^{e_{i}})}-\sqrt{\bar{f}(\eta)}\Big]^{2}\nu_{\alpha}(d\eta)
\end{equation*}
\begin{equation*}
+C(d)N
\epsilon\sum_{\substack{a_{1}\in{\wedge_{l}(0)}\\{..}\\{a_{2^d}\in{\wedge_{l}(0)}}}}\int_{\Omega_{0}^c}1_{\{a_{1}=x_{1},..,a_{2^d}=x_{2^d}\}}
\sum_{\tilde{e}_{i}}c(\tilde{e}_{i},\eta)\Big[\sqrt{\bar{f}(\eta^{\tilde{e}_{i}})}- \sqrt{\bar{f}(\eta)}\Big]^{2}\nu_{\alpha}(d\eta),
\end{equation*}
where $\{e_{i}\}_{i}$ denotes the bonds that we use inside $\wedge_{l}(0)$ when taking the $2^d$ particles whose initial positions are
$x_{1},..x_{2^d}$, close to the particle at the site $0$, by the jumps of the exclusion process, while $\tilde{e}_{i}$ corresponds to the bonds
used by the generator $\mathcal{L}_{P}$ when performing the rest of the path and $c(\tilde{e}_{i},\eta)$ denotes the corresponding jump rate.
Since there are $l^{d2^d}$ chances for the initial positions $x_1,\dots x_{2^d}$, we can bound last expression by:
\begin{equation*}
l^{d2^d}C(d)l\sum_{x,x+1\in{\wedge_{l}(0)}}I_{x,x+1}^{S}(\bar{f},\nu_{\alpha}) +l^{d2^d}C(d)\epsilon
N\sum_{x,x+1\in{\wedge_{l}(0)}}I_{x,x+1}^{P}(\bar{f},\nu_{\alpha}).
\end{equation*}

By the equality (\ref{relation I}) and the bound on the Dirichlet form of $f$, it holds that:
\begin{equation*}
\forall{x\in{\mathbb{T}_{N}^{d}}}:\\ \quad\ I_{x,x+1}^{S}(\bar{f},\nu_{\alpha})\leq{\frac{C}{N^\theta}},
\\ \quad\
I_{x,x+1}^{P}(\bar{f},\nu_{\alpha})\leq{\frac{C}{N^{2}}},
\end{equation*}
which together with $l\leq{\epsilon N}$, ends the proof.
\end{proof}

By the definition of the Dirichlet form $D_{\theta}^{2,l}(f,\nu_{\alpha}^{2,l})$ in (\ref{dirichletformoff_{0,y}}), together with
(\ref{firstpartofdirichletform}) and the previous Lemma, we can restrict to densities $f$ that satisfy
\begin{equation*}
D^{2,l}_{\theta}(f,\nu_{\alpha}^{2,l})\leq{C(d,l) \epsilon^{\theta}}.
\end{equation*}

So, to conclude the proof of the Two-blocks estimate it is enough to show that
\begin{equation*}
\limsup_{l\rightarrow{+\infty}}\\\limsup_{\epsilon\rightarrow{0}}
\\\sup_{D_{\theta}^{2,l}(f,\nu_{\alpha}^{2,l})\leq{C(l)\epsilon^\theta}}
\int_{\Omega^{c}_{0,y}}|\xi_{1}^{l}(0)-\xi_{2}^{l}(0)|f(\xi)\nu_{\alpha}^{2,l}(d\xi)=0,
\end{equation*}
where the supreme is carried over densities with respect to $\nu_{\alpha}^{2,l}$. Note that the parameter $\epsilon$ is appearing only on the
bound of the Dirichlet form.  The proof now follows the same lines as in the case of the Zero-Range process in \cite{K.L.}, by starting to
decompose the Dirichlet form along the hyperplanes and applying the equivalence of ensembles. For the detailed strategy see section 5.5 of
\cite{K.L.}.

\section{Spectral Gap}
\label{gap}

In this section we analyze the magnitude of the spectral gap for the
one dimensional generators on finite boxes, ${\mathcal{L}}_{P,\Lambda_N}$ and ${\mathcal{L}}^N_{\theta,\Lambda_N}$, which have been defined in section 2.
Note that the results below on the scaling of the spectral gap with the lattice size have not been used in the previous sections to derive the hydrodynamic limit. This was possible thanks   to the fact that we were considering a gradient choice of the rates. The following analysis of the spectral gap
can be regarded as a first step towards the analysis of the non gradient version of our models, e.g. for the choice
\begin{eqnarray*}
c'(x,x+e_j,\eta)=\left\{
\begin{array}{rl}
1  & \mbox{if~} \eta(x-e_j)+\eta(x+2 e_j)\geq 1\\
0 & \mbox{otherwise}
\end{array}.
\right.
\end{eqnarray*}
which is a non gradient version with the same kinetic constraints as our original choice (\ref{rates}),  namely $c'(x,x+e_i,\eta)=0$ if and only if $c(x,x+e_i,\eta)=0$.

For simplicity in the following we drop $\Lambda_N$ from our notation. The main ingredient of our
 proofs will be a comparison with the spectral gap $\lambda({\mathcal{L}}_{LR})$ for the {\sl long range exclusion process} and the
 use of path arguments. Let us start by recalling the definition of ${\mathcal{L}}_{LR}$ and the result in \cite{Q} for its spectral gap.
 The action of ${\mathcal{L}}_{LR}$ on local functions $f:\Sigma_{N,k}\to\mathbb R$ is given by

\begin{equation} \label{longrangeexclusion}
(\mathcal{L}_{LR}f)(\eta)=\sum_{x,y\in{\Lambda_{N}}}\frac{1}{N}\eta(x)(1-\eta(y))(f(\eta^{x,y})-f(\eta)),
\end{equation}
where $\eta^{x,y}$ as defined in (\ref{etaxy}).

Consider the Dirichlet form $\mathfrak{D}_{LR}$ associated to $\mathcal{L}_{LR}$, with respect to the uniform measure $\nu_{N,k}$, given
explicitly by:
\begin{equation} \label{dirichletformLRange}
\mathfrak{D}_{LR}(f,\nu_{N,k})=\frac{1}{N}\sum_{x,y\in{\Lambda_{N}}}\int_{\Sigma_{N,k}}(f(\eta^{x,y})-f(\eta))^{2}\nu_{N,k}(d\eta).
\end{equation}

Quastel in \cite{Q}, obtained that
\begin{equation}
\textbf{Var}(f,\nu_{N,k})\leq{\mathfrak{D}_{LR}(f,\nu_{N,k})}, \label{eq:sgaplongrange}
\end{equation}
by computing precisely the eigenvalues of the generator $\mathcal{L}_{LR}$. In order to prove our results (Proposition \ref{spectralgaplarge} and \ref{spectralgapsmall}) for the spectral gap of the process with
generator $\mathcal{L}_{\theta}^N$ $\forall k$ and with generator
$\mathcal{L}_{P}$ for $k>1/3$, we proceed in two steps.

Fix an integer $k$ such that the density is restricted to $\rho=\frac{k}{N}>\frac{1}{3}$. With this restriction, for each
$\eta\in{\Sigma_{N,k}}$, there exists a couple of particles whose distance is smaller or equal to two and $\Sigma_{N,k}$ is irreducible. Furthermore we will show that there exists a constant $C$ that does not depend
on $N$ nor $k$ such that
\begin{equation}
\mathfrak{D}_{LR}(f,\nu_{N,k})\leq{\frac{C\rho}{\rho-1/3}N^2\mathfrak{D}_{P}(f,\nu_{N,k})}, \label{comparedirichletformlargedensity}
\end{equation}
where $\mathfrak{D}_{P}(f,\nu_{N,k})$ is the Dirichlet that corresponds to $\mathcal{L}_{P}$.

Last result together with (\ref{eq:sgaplongrange}) allows to conclude the following Poincar\'e inequality:

\begin{proposition} \label{thsgaplargedensity}
Fix $k>N/3$. For the Markov process with generator $\mathcal{L}_{P,N}$ and for every $f\in{L^{2}(\nu_{N,k})}$, there exists a constant $C$ not
depending on $N$ nor $k$ such that:
\begin{equation*}
\textbf{Var}(f,\nu_{N,k})\leq{\frac{C\rho}{\rho-1/3}N^2\mathfrak{D}_{P}(f,\nu_{N,k})}.
\end{equation*}
\end{proposition}

The second inequality in the result of Proposition \ref{spectralgaplarge} is an immediate consequence of last result. The first inequality
follows from the fact that for any function $f$ it holds that $\mathfrak{D}_{P}(f,\nu_{N,k})\leq\mathfrak{D}_{\theta}(f,\nu_{N,k})$.

The case in which $\rho=\frac{k}{N}\leq{\frac{1}{3}}$ is more demanding. Since the hyperplanes $\Sigma_{N,k}$ are no longer irreducible for
${\mathcal{L}}_P$, the corresponding spectral gap is zero. A natural issue is determining the magnitude of the spectral gap for
${\mathcal{L}}_P$ on the restricted irreducible set of configurations with $k$ particles and at least one couple of particles at distance at
most two. In this case the invariant measure is no longer the same one as for the long range jumps and the comparison with the latter process is
no more useful. Instead, we will here consider for $k/N\leq 1/3$ the  spectral gap for the modified generator ${\mathcal{L}}_{\theta}^N$ which,
thanks to the addition of the exclusion part, is  ergodic on the hyperplanes $\Sigma_{N,k}$ for any $k$ and reversible with respect to
$\nu_{N,k}$. By comparing the Dirichlet form of $\mathcal{L}_{\theta}^N$ with the one of $\mathcal{L}_{LR}$, we show that:
\begin{proposition} \label{thsgapsmalldensity}
Fix $k\leq{N/3}$. For the Markov process with generator given by $\mathcal{L}_{\theta}^N$, with $\theta=1$ and for every
$f\in{L^{2}(\nu_{N,k})}$, there exists a constant $C$ not depending on $N$ nor $k$ such that:
\begin{equation*}
\textbf{Var}(f,\nu_{N,k})\leq{\frac{C}{\rho^{2}}N\mathfrak{D}_{S}(f,\nu_{N,k})+\frac{C}{\rho}N^2\mathfrak{D}_{P}(f,\nu_{N,k})}.
\end{equation*}
\end{proposition}
Proposition \ref{spectralgapsmall}
 is an immediate consequence of Proposition \ref{thsgapsmalldensity}.

\subsection{Proof of Proposition \ref{thsgaplargedensity}}

\quad \vspace{0.1cm}

Fix an integer $k$ and suppose that $\rho=\frac{k}{N}>\frac{1}{3}$. The Dirichlet form associated to $\mathcal{L}_{P}$ is given explicitly by
\begin{equation*}
\mathfrak{D}_{P}(f,\nu_{N,k})=\frac{1}{2}\sum_{\substack{x,y\in{\Lambda_{N}}\\|x-y|=1}}\int_{\Sigma_{N,k}}c(x,y,\eta)(f(\eta^{x,y})-f(\eta))^{2}\nu_{N,k}(d\eta)
\end{equation*}
where $c(x,y,\eta)$ and $\eta^{x,y}$ were defined on (\ref{rates}) and (\ref{etaxy}), respectively. We have seen above that it is enough to show
(\ref{comparedirichletformlargedensity}).

The idea consists in expressing the exchange $\eta^{x,y}$, for each $x,y\in{\Lambda_{N}}$ by means of a sequence of allowed jumps of
$\mathcal{L}_{P}$. Consider for example the case $x<y$ and $\eta(x)=1$, $\eta(y)=0$. The restriction $\rho>\frac{1}{3}$ guarantees that there
exists a couple of sites $(a,b)$ both in $\Lambda_{N}$ such that $\eta(a)=\eta(b)=1$ and $|a-b|\leq{2}$. Once the couple has been chosen
we shift it
close to sites $x-2,x-1$.
We have thus reached a configuration with
$\eta(x-2)=\eta(x-1)=\eta(x)=1$, and we now
shift the three particles to $(y-3,y-2,y-1)$.
The allowed path which performs the shift of one step
to the right is the composition of the following
three basic steps:
$\eta\to \eta_1=\eta^{x,x+1}$, $\eta_1\to\eta_2=\eta_1^{x-1,x}$ and
$\eta_2\to\eta_3=\eta_2^{x-2,x-1}$.
When after a proper number of one step shifts
we reach the configuration $\eta_n$ with
$\eta(y-3)=\eta(y-2)=\eta(y-1)=1$, we can
perform the exchange $\eta_n\to\eta_{n+1}=\eta_n^{y-1,y}$ which corresponds to dropping the particle at $y$. Then we make a similar backward path, this time shifting the two particles plus the vacancy untill bringing the vacancy in $x$ and finally we bring back the two particles to their original position $(a,b)$.

As we did for Lemma \ref{pathlemma}, we have to take into account the entropic term corresponding to the possible initial positions of the
couple of particles $(a,b)$. Some care is required at this point, since using the rough bound $N$ for this original position would lead to an
additional factor $N$ in the inequality \ref{comparedirichletformlargedensity}. Let us start by presenting  a lower bound for the number of
couples of particles at distance at most two \quad \vspace{0.5cm}

\textbf{Lower bound for the number of couples for $k>N/3$}

\quad\ Fix a configuration $\eta\in{\Sigma_{N,k}}$. Let $\mathcal{B}$ be the set of sites which are occupied and such that for each of them
there is a particle at distance at most two, $\mathcal B=\{x: \eta(x)=1, \eta(x+1)+\eta(x-1)+\eta(x+2)+\eta(x-2)\geq{1}\}$.
Let also $\mathcal A$ denote the remaining occupied sites $\mathcal A=\{x:\eta(x)=1, x\in \Lambda_N\setminus \mathcal B\}$. The following holds:
\begin{equation*}
k=|\mathcal{A}|+|\mathcal{B}|
\end{equation*}
\begin{equation*}
N\geq{3(|\mathcal{A}|-1)+1+|\mathcal{B}}|,
\end{equation*}
where for a set $S$, $|S|$ denotes the cardinality of the set.

The first equality is obvious. In order to establish the second property we use the fact that, if $x$ and $y$ belong to $\mathcal A$, then $|x-y|\geq 3$.
The inequality follows by organizing the $k$ particles in the closest configuration which does not change the value of $|\mathcal A|$ and  $|\mathcal B|$.
The minimum number of couples of particles at distance at most two and such that there are not common sites among different couples is  $|\mathcal{B}|/2$, see e.g. the figure below
\begin{figure}[htp]
  \includegraphics[width=8cm]{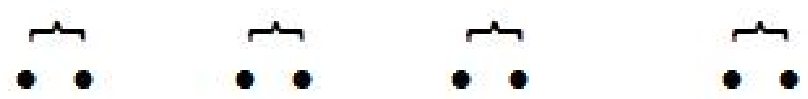}\\
\end{figure}
in which there are $8$ neighboring particles but just $4$ couples.
This, together with the above established relations among $k,N,|\mathcal B|$, gives the following lower bound for the number of couples
\begin{equation} \label{couplelowerbound}
\sum_{z\in{\Lambda_{N}}}\eta(z)\eta(z+1)+\eta(z)\eta(z+2)\geq{|\mathcal{B}|/2\geq{\frac{3}{4}(k-N/3-2/3)}}.
\end{equation}

Introducing this term in the Dirichlet form of the long range exclusion (\ref{dirichletformLRange}) we bound it by:
\begin{equation}
\frac{2}{N|\mathcal{B}|}\sum_{x,y\in{\Lambda_{N}}}\sum_{\eta\in{\Sigma_{N,k}}}\sum_{z\in{\Lambda_{N}}}\sum_{l=1}^{2}\eta(z)\eta(z+l)
(f(\eta^{x,y})-f(\eta))^{2}\nu_{N,k}(\eta). \label{sgap1}
\end{equation}
Let us now consider a term in the sum above
\begin{equation*}
\eta(z)\eta(z+l)\Big(f(\eta^{x,y})-f(\eta)\Big)^{2}
\end{equation*}
and let $x-(z+1)=n$ and $y-x=m$.
By using the construction sketched above for a the possible path  which connects $\eta$ to $\eta^{x,y}$ via allowed elementary exchanges, it is
possible to define a sequence $\eta_i$ for $1\leq i\leq \gamma_z^{x,y}$ with $\gamma_z^{x,y}=5(m-1)+4(n-1)+2$ with the following properties:
$\eta_1=\eta$, $\eta_{\gamma_z^{x,y}}=\eta^{x,y}$ and $\forall i$ there exists $x_i\in\Lambda_N$ such that $\eta_i=\eta_{i-1}^{x(i),x(i)+1}$ and
$c(x(i),x(i)+1,\eta_{i-1})>0$, namely the exchanges are permitted for the generator $\mathcal{L}_{P}$.
Therefore we can rewrite each $\eta(z)\eta(z+l)(f(\eta^{x,y})-f(\eta))^2$ as the telescopic sum:
\begin{equation*}
\eta(z)\eta(z+l)\Big(f(\eta^{x,y})-f(\eta)\Big)^{2}=\eta(z)\eta(z+l)\Big(\sum_{i=1}^{\gamma^{x,y}_{z}-1}f(\eta_n)-f(\eta_{n-1})\Big)^2
\end{equation*}
By using this equality together with $c(x(i),x(i)+1,\eta_{i-1})>0$, the fact that the size of the path $\gamma_{x,y}^z$ is of $O(N)$ and
applying Cauchy-Schwarz inequality we can finally bound (\ref{sgap1}) from above by:
\begin{equation*}
\frac{CN}{N|\mathcal{B}|}\sum_{x,y\in{\Lambda_{N}}}\sum_{\eta\in{\Sigma_{N,k}}}\sum_{z\in{\Lambda_{N}}}\sum_{l=1}^{2}\eta(z)\eta(z+l)
\sum_{\tilde{e}_{i}}c(i,i+1,\eta)(f(\eta^{\tilde{e}_{i}})-f(\eta))^{2}\nu_{N,k}(\eta),
\end{equation*}
where $\tilde{e}_{i}=\{i,i+1\}$ denotes one bond that we have used when performing the path that takes the particle from $x$ to $y$ using the
couple at the sites $z$ and $z+1$ and $C$ is a constant independent on $N$ and $k$.
Estimating over all $x$ and $y\in{\Lambda_{N}}$, last expression is bounded above by
\begin{equation*}
\frac{CN^3}{N|\mathcal{B}|}\sum_{\eta\in{\Sigma_{N,k}}}\sum_{z\in{\Lambda_{N}}}\eta(z)\sum_{x}c(x,x+1,\eta)(f(\eta^{x,x+1})-f(\eta))^{2}\nu_{N,k}(\eta).
\end{equation*}
By the estimate (\ref{couplelowerbound}) and since the number of particles is $k$, last expression is bounded above by
\begin{equation*}
\frac{C\rho}{\rho-1/3}N^2\mathfrak{D}_{P}(f,\nu_{N,k}).
\end{equation*}

\vspace{0.3cm}

\subsection{Proof of Proposition \ref{thsgapsmalldensity}}

\quad\ \vspace{0.3cm}

Now, fix an integer $k$ such that $\rho=\frac{k}{N}\leq{\frac{1}{3}}$. In this case, we are going to show that
\begin{equation}
\mathfrak{D}_{LR}(f,\nu_{N,k})\leq{\frac{C}{\rho^{2}}N\mathfrak{D}_{S}(f,\nu_{N,k})+\frac{C}{\rho}N^2\mathfrak{D}_{P}(f,\nu_{N,k})},
\label{comparedirichletformsmalldensity}
\end{equation}
which is enough to conclude. As before, the idea consists in expressing a path from $x$ to $y$ using the admissible jumps of
$\mathcal{L}_{\theta}^N$.

For this choice of $k$ we are no more guaranteed that there exist two particles at distance at most two, which was a key ingredient to construct the path in the high density regime. The idea will be to make use of the simple exclusion jumps to construct such mobile clusters and then proceed as before in a similar way as we did in
in the proof of Lemma \ref{pathlemma}.

Fix a distance $j$ and
denote by $\mathcal{B}_{j}$ the set of particles at distance at most $j$, $\mathcal B_j=\{x: \eta(x)=1, \exists l\in(-j,-1)\cup(1,j) ~s.t.~
\eta(x+l)=1\}$
 and by $\mathcal{A}_{j}$ the remaining particles.
 Then, the following holds:
\begin{equation*}
k=|\mathcal{A}_{j}|+|\mathcal{B}_{j}|
\end{equation*}
\begin{equation*}
N\geq{(j+1)(|\mathcal{A}_{j}|-1)+1+|\mathcal{B}_{j}|}
\end{equation*}
This inequality is obtained in the same manner as before considering that now the minimum distance that the $\mathcal{A}_{j}$-particles have to be is $j+1$.

As before, the minimum number of couples that one can have is $|\mathcal{B}_{j}|/2$. By simple computations we obtain the lower bound
\begin{equation} \label{2couplelowerbound}
\sum_{z\in{\Lambda_{N}}}\sum_{l=1}^{j}\eta(z)\eta(z+l)\geq\frac{|\mathcal{B}_{j}|}{2}\geq{\frac{j+1}{2j}\Big(k-\frac{N}{j+1}-\frac{j}{j+1}}\Big).
\end{equation}

Introducing this inequality in the Dirichlet form of the long range exclusion (\ref{dirichletformLRange}), we bound it from above by:
\begin{equation*}
\frac{C}{N|\mathcal{B}_{j}|}\sum_{x,y\in{\Lambda_{N}}}\sum_{\eta\in{\Sigma_{N,k}}}\sum_{z\in{\Lambda_{N}}}\sum_{l=1}^{j}\eta(z)\eta(z+l)
(f(\eta^{x,y})-f(\eta))^{2}\nu_{N,k}(\eta).
\end{equation*}
For each configuration $\eta$ and each choice $x$, $y$, $z$ and $l$, we can now construct a path which first bring together the two particles in
$z$ and $z+l$ by using the jumps of the simple exclusion and then uses this mobile cluster to perform the exchange of occupation variables in
$x$ and $y$,
as in the proof of Lemma \ref{pathlemma}. Here $\mathcal{L}_{S}$, also denotes the generator of the Symmetric Simple Exclusion process
restricted to the box $\Lambda_{N}$, given on local functions by
\begin{equation*}
(\mathcal{L}_{S}f)(\eta)=\sum_{\substack{x,y\in{\Lambda_{N}}\\|x-y|=1}}\frac{1}{2}\eta(x)(1-\eta(y))(f(\eta^{x,y})-f(\eta)).
\end{equation*}

Since the jumps of the exclusion are used just to put the neighboring particles at a distance equal to two, the size of the path for this
process is $O(j)$. With this purpose, write $f(\eta^{x,y})-f(\eta)$ as a telescopic sum, use the elementary inequality $(x+y)^2\leq{2x^2+2y^2}$
and then the Cauchy-Schwarz inequality to bound last expression by:
\begin{equation}
\frac{Cj}{N|\mathcal{B}_j|}\sum_{x,y\in{\Lambda_{N}}}\sum_{\eta\in{\Sigma_{N,k}}}\sum_{z\in{\Lambda_{N}}}\sum_{l=1}^{j}\eta(z)\eta(z+l)
\sum_{e_{i}}(f(\eta^{{e}_{i}})-f(\eta))^{2}\nu_{N,k}(\eta) \label{exclusionjumps}
\end{equation}
\begin{equation}
+\frac{CN}{N|\mathcal{B}_j|}\sum_{x,y\in{\Lambda_{N}}}\sum_{\eta\in{\Sigma_{N,k}}}\sum_{z\in{\Lambda_{N}}}\sum_{l=1}^{j}\eta(z)\eta(z+l)
\sum_{\tilde{e}_{i}}c(i,i+1,\eta)(f(\eta^{\tilde{e}_{i}})-f(\eta))^{2}\nu_{N,k}(\eta), \label{porousjumps}
\end{equation}
where $e_{i}$ denotes the bonds that we have used when bringing the neighboring particles at a distance equal to two, while
$\tilde{e}_{i}=\{i,i+1\}$ denotes the bonds we have used when performing the remaining part of the path that takes the particle from $x$ to $y$.
Note, that there is a factor $j$ multiplying the first expression which comes from the size of the path for the jumps of the exclusion, while
for the other process the size of the path is of $O(N)$.

First we deal with the jumps that concerns $\mathcal{L}_{P}$. As before,
we bound (\ref{porousjumps}) from above by
\begin{equation*}
\frac{CjN^3}{N|\mathcal{B}_{j}|}\sum_{\eta\in{\Sigma_{N,k}}}\sum_{z\in{\Lambda_{N}}}\eta(z)\sum_{x}c(x,x+1,\eta)(f(\eta^{x,x+1})-f(\eta))^{2}\nu_{N,k}(\eta),
\end{equation*}
and since $\eta\in{\Sigma_{N,k}}$, we obtain the bound:
\begin{equation*}
\frac{Cj\rho}{\rho-1/j}N^2\mathfrak{D}_{P}(f,\nu_{N,k}).
\end{equation*}
Now, we bound (\ref{exclusionjumps}), by
\begin{equation*}
\frac{CjN^2}{N|\mathcal{B}_{j}|}\sum_{\eta\in{\Sigma_{N,k}}}\sum_{z\in{\Lambda_{N}}}\sum_{l=1}^{j}\eta(z)
\sum_{x}(f(\eta^{x,x+1})-f(\eta))^{2}\nu_{N,k}(\eta).
\end{equation*}
Since the number of particles is $k$, we can bound last expression by:
\begin{equation*}
\frac{Cj^2\rho}{(\rho-1/j)}N\mathfrak{D}_{S}(f,\nu_{N,k}).
\end{equation*}

Reorganizing these facts together we obtain that:
\begin{equation*}
\mathfrak{D}_{LR}(f,\nu_{N,k})\leq{\frac{Cj^2\rho}{(\rho-1/j)}N\mathfrak{D}_{S}(f,\nu_{N,k})+\frac{Cj\rho}{\rho-1/j}N^2\mathfrak{D}_{P}(f,\nu_{N,k})}.
\end{equation*}
Optimizing over $j$, (\ref{comparedirichletformsmalldensity}) follows.

\begin{remark}
For sake of simplicity we have presented the spectral gap results only in the one dimensional setting. By a proper modification of the path
arguments and an accurate estimate on the minimal number of  mobile clusters it is possible to obtain for $d>1$ an analogous result as the one
in Proposition \ref{spectralgaplarge} if the density $\rho=k/N$ is such that $k>C(d)(N/3)^d$.
\end{remark}

\begin{remark}
\label{remarkgap}
Taking for instance the density fluctuations field as defined in (\ref{densityfield}) and the reference measure the Bernoulli product measure
$\nu_{\rho}$, by simple computations we obtain that
\begin{equation*}
\textbf{Var}(\mathcal{Y}^{N}_{t}(H),\nu_{\rho})=\rho(1-\rho)||H||_{2}^{2},
\end{equation*}
while the Dirichlet form corresponding to $\mathcal{L}_{P}$ equals to:
\begin{equation*}
\mathfrak{D}_{P}(\mathcal{Y}^{N}_{t}(H),\nu_{\rho})=\frac{1}{N^{2}}\rho^{2}(1-\rho)||H'||_{2}^{2}.
\end{equation*}

So, if we consider $H$, such that $||H||_{2}^{2}=||H'||_{2}^{2}$, then:
\begin{equation*}
\textbf{Var}(\mathcal{Y}^{N}_{t}(H),\nu_{\rho})=\frac{N^2}{\rho} \mathfrak{D}_{P}(\mathcal{Y}^{N}_{t}(H),\nu_{\rho}),
\end{equation*}
which implies that the spectral gap $\lambda_{N}(\mathcal{L}_{P,\Lambda_{N}})\leq{\frac{\rho}{N^{2}}}$. This is in agreement with the bound that
we have obtained in (\ref{comparedirichletformsmalldensity}), when considering the spectral gap with respect to the uniform measure.
\end{remark}

\end{document}